\documentclass[12pt,spanish]{article}
\usepackage{graphics,psfrag,graphicx,epsfig}
\usepackage{amsmath,amsfonts,amssymb,amstext,amsthm,amscd,mathrsfs}
\usepackage[spanish,english]{babel}
\usepackage{breqn}

\theoremstyle{definition}
\newtheorem{defi}{Definition}
\newtheorem{obs}{Observation}

\newtheorem{rmk}{Remark}

\theoremstyle{plain}
\newtheorem{theo}{Theorem}

\newtheorem{prop}{Proposition}
\newtheorem{lema}{Lemma}

\newcommand{\C}{\mathbb{C}}
\newcommand{\R}{\mathbb{R}}
\newcommand{\Z}{\mathbb{Z}}
\newcommand{\T} {\mathbb{T}}

\newcommand{\vd} {\mathbf M}

\newcommand{\s}{\mathbb{S}}

\newcommand{\tange} {\mathbf{T}}
\newcommand{\noi}{{\noindent}}
\newcommand{\Lam}{\mathbf{\Lambda}}

\newcommand{\K} {\mathcal{K}}

\newcommand{\cv}{\mathcal{X}}
\newcommand{\cvd}{\mathcal{Y}}
\newcommand{\U} {\mathcal{U}}

\newcommand{\V} {\mathcal{V}}
\newcommand{\ls}{\mathcal{L}}

\newcommand{\lam}{\pmb{\lambda}}

\newcommand{\MLamU}{{\mathcal M}^{^{(\mathbf{\Lambda},m,n)}}_{_1}}
\newcommand{\MLam}{{\mathcal M}^{^{(\mathbf{\Lambda},m,n)}}}
\newcommand{\Mg}{\vd^{^{(\mathbf{\Lambda},1,n+s)}}_{_1}}

\newcommand{\Mgo}{\vd^{^{(\mathbf{\Lambda},1,n+s)}}}
\newcommand{\MGo}{\vd^{^{(\mathbf{\Lambda},m,n)}}}
\newcommand{\MG}{\vd^{^{(\mathbf{\Lambda},m,n)}}_{_1}}

\newcommand{\Mgu}{\vd^{^{(\mathbf{\Lambda},1,n+1)}}_{_1}}
\newcommand{\MgU}{\vd^{^{(\mathbf{\Lambda},1,n+s-1)}}_{_1}}
\newcommand{\ev}{\mathcal{H}}

\begin{document}

\title {Moment-angle manifolds, intersection of quadrics and higher dimensional contact manifolds\thanks{Partially supported by project PAPIIT-DGAPA IN100811 and project CONACyT 129280.}}
\author{Yadira Barreto\\ and \\ Alberto Verjovsky }
\date{}

\maketitle

\begin{abstract}
\noi We construct new examples of contact manifolds in arbitrarily large  dimensions. These manifolds which we call {\it moment-angle mani\-folds of mixed type}, are closely related to the classical moment-angle manifolds. 

\end{abstract}

\noi {\bf Key Words:} Moment-angle manifolds of mixed type,  moment-angle manifolds, po\-si\-tive confoliations, Poisson and contact structures.\\

\noi  {\bf AMS subject classification:} Primary 53D10, 53D35. Secondary 57R17, 53D17, 14M25.

\section{Introduction.}

\noi The topology of intersection of quadrics in $\C^n$ of the form:

\begin{eqnarray}\label{E1}
H\big(z_1,\dots,z_n\big):=\sum_{j=1}^n\lam_j|z_j|^2=0,\\
\nonumber \rho\big(z_1,\dots,z_n\big):=\sum_{j=1}^n|z_j|^2=1,
\end{eqnarray}

\noi where $\lam_j=\big(\lambda^1_{_j},\dots,\lambda^m_{_j}\big)\in\C^m$ for $m\geqslant 1$, $n>3$ such that  $n>2m$  and the configuration $\Lam=\big(\lam_1,\dots,\lam_n\big)$ satisfying an {\it admissibility condition},  has been studied by di\-fferent authors. For instance, S. L\'opez de Medrano and A. Verjovsky in \cite{LV} studied the case  $m=1$. In this case, these intersections are diffeomorphic  to a triple product of spheres or to the connected sum of manifolds which are the product of two spheres of varying dimensions. The genera\-li\-zation to $m>1$ is due to L. Meersseman in \cite{Meer} in which it is remarked their very beautiful and rich geometry and topology (see also the paper by S. Gitler and S. L\'opez de Medrano \cite{LoGli}).\\

\noi For every $m\geq1$ these manifolds are known as \emph{moment-angle manifolds} and we denote them by $\MLamU$. They are principal circle bundles over compact complex manifolds which are never K\"ahler  except in few cases (see \cite{LV} and \cite{Meer}). These manifolds have been extensively studied by V. Buchstaber and T. Panov, among others (see \cite{BBCG}, \cite{BP} and \cite{DJ} for instance).\\

\noi Moment-angle manifolds are also interesting since they have the structure of open books, and they admit different geometric structures such as: almost contact structures and quasi contact structures (see  \cite{LoGli} and \cite{BLV}).   
Moment-angle manifolds admit a locally free action of $\R^{2k}$ ($1\leq{k}\leq{m}$) (see \cite{Meer}) and therefore the orbits determine a symplectic foliation in a natural way and as a consequence  they are regular Poisson manifolds (see \cite{IV}).\\

 \noi In \cite {MV} it was shown that for every $m\geq1$ the compact complex manifold (which in general is non K\"ahler), obtained as the orbit space of the principal circle action on a moment-angle manifold, admits a locally-free holomorphic action of $\C^m$.  Furthermore the $m$-dimensional holomorphic foliation $\mathcal{F}_h$ obtained by this action is a transversally K\"ahler foliation (see \cite{LN}).\\
 
 \noi  Under a rationality condition on the configuration of vectors $\Lam$  the leaf space of  $\mathcal{F}_h$ is Hausdorff and it is a toric algebraic variety or a toric algebraic or\-bifold. In addition, {\it eve\-ry toric manifold (or orbifold with quotient singulari\-ties) is obtained by this construction}. The set of configurations satisfying the rationality condition is dense in the set of configurations.\\
 
 \newpage
 
 \noi A fundamental paper intertwining real quadrics in $\C^n$, complex manifolds and combinatorics of polytopes is the paper of L. Meersseman and F. Bosio  \cite{BM}  in {\it Acta Mathematica}. \\

 \noi In section \ref{S3} we define a generalization of moment-angle manifolds, which we call \emph{moment-angle manifolds of mixed type}\footnote {These manifolds are called  \emph{moment-angle manifolds of mixed type}, using the notation of Davis and Januszkiewicz \cite{DJ}, since they admit the action $\big(\Z/2\Z\big)^s \times \big(\s^1\big)^n$ for $m=1$ and $\big(\Z/2\Z\big)^m \times \big(\s^1\big)^n$ for $m>1$.} obtained by the intersection of the following non coaxial quadrics (see \cite{GL} for case $m=1$): 
 \begin{itemize}
\item For $m=1$ and $s\geq1$:
 $$
 \sum_{r=1}^s w^2_r+\sum_{j=1}^n\lambda_j|z_j|^2=0,\quad\quad \sum_{r=1}^s|w_r|^2+\sum_{j=1}^n|z_j|^2=1,
 $$
 \noi where $\lambda_j\in\C$ for all $j\in\{1,\dots,n\}.$

\item For $m>1$:
$$
\pmb{w}^2+\sum_{j=1}^n\lam_j|z_j|^2=0,\quad \quad \sum_{r=1}^m|w_r|^2+\sum_{j=1}^n|z_j|^2=1,$$
\noi where $\pmb{w}^2:=\big(w_1^2,\dots,w_m^2\big)\in\C^m$ with $w_r\in\C$ ($1\leq{r}\leq{m}$) and $\lam_j=\big(\lambda^1_{_j},\dots,\lambda^m_{_j}\big)\in\C^m$ for all $j\in\{1,\dots,n\}$.
\end{itemize}

\noi In the first case, the manifolds obtained are connected sums of pro\-ducts\- of two spheres  (see \cite{GL}) and also admit an open book decomposition (see \cite{LoGli}, \cite{SLM1} and \cite{BLV}).\\

\noi Observe that when $s=0$ and $\pmb{w}^2=(0,\dots,0)\in\C^m$,  in the first and second case, respectively, we obtain a moment-angle manifold $\MLamU$.\\

\noi  In section \ref{S3} we introduce an idea originally conceived by S. J Altschuler in \cite{Alt} and generalized by S. J. Altschuler and F. Wu in \cite{AltWu} about conductive confoliations, i.e.,  confoliations that are able to conduct  \emph{contactness}  to all points of a manifold via paths with certain characteristics (essentially by Legendrian curves when the 1-form is a contact form).\\ 

\noi We will show that the canonical $1$-form in $\C^N$ given by
$$
\alpha=i\sum_{j=1}^N\big(z_jd\bar{z}_j-\bar{z}_jdz_j\big)
$$

\noi when restricted to the tangent bundle of a moment-angle manifold is non\-tri\-vial and $\big[\ker(\alpha)\cap\ker(d\alpha)\big]$ is a regular Poisson structure whose symplectic foliation is tangent to a locally free action of an $2m$-torus $\T^{2m}=\s^1\times\dots\times\s^1$. On the other hand, for a moment-angle manifold of mixed type this real $1$-form can be deformed into a contact form, by an arbitrarily small $C^{^\infty}$ perturbation.\\
 
\noi Our construction is in some sense explicit  since is the instantaneous  diffusion through the heat flow of an explicit 1-form which is a positive confoliation.\\

\noi In section \ref{S4} we give some concluding remarks indicating that moment-angle manifolds of mixed type admit an action
of a torus such that the orbit space is a polytope which is the product of two simple polytopes and as a consequence their topology can be determined by
combinatorics and algebra associated to polytopes (see \cite{BP}, \cite{BP2} and \cite{BBCG}).

\section{Moment-angle manifolds.}\label{S2}

\noi Let  $m$ and $n$ be two positive integers such that $m\geqslant 1$, $n>3$ and $n>2m$. Let $\Lam=\big(\lam_1,\dots,\lam_n\big)$  be an $n$-tuple of vectors in  $\C^m$,  with $\lam_j=\big(\lambda^1_{_j},\dots,\lambda^m_{_j}\big)$ for all  $j\in\{1,\dots,n\}$. Let $\ev(\Lam)$ the convex hull of $\Lam$ in $\C^m$.

\begin{defi}\label{D1}\cite{Meer}
An  \emph{admissible configuration} is an $n$-tuple $\Lam$ satisfying the fo\-llo\-wing:

\begin{enumerate}
\item  {\bf Siegel condition:} $0\in \ev(\Lam)$.
\item {\bf Weak hyperbolicity condition:} for every $2m$-tuple of integers\\ $\left\{j_1,\dots,j_m\right\}$ such that $1\leqslant j_1<\dots<j_{2m}\leqslant n$, we have  
$$
0\notin \ev\big(\lam_{j_1},\dots,\lam_{j_{2m}}\big).$$
\noi In other words: The convex polytope $\ev(\Lam)$ contains $0$, but any hyperplane passing through $2m$ vertices does not contains $0$.
\end{enumerate}
\end{defi}

\noi Every admissible configuration satisfies the following regularity property:

\begin{lema}\cite[Lemma I.1]{Meer}\label{L1}
Let $\lam'_i=\big(\lam_i,1\big)\in\C^{m+1}$ with $i\in\{1,\dots,n\}$. For all set of integers $J$ between $1$ and $n$ such that $0\in\ev\big((\lam_j)_{j\in J}\big)$, the complex rank of the matrix whose columns are the vectors $\big(\lam'_j\big)_{j\in J}$ is equal to $m+1$, therefore maximal.
\end{lema}

\noi Let now $m\geq1$, $n>3$ and $n>2m$.\\

\noi Let $\Lam=\big(\lam_1,\dots,\lam_n\big)$ be an $n$-tuple of vectors in $\C^m$, with $\lam_j=\big(\lambda^1_{_j},\dots,\lambda^m_{_j}\big)$ for all $j\in\{1,\dots,n\}$. Consider the following system of  quadratic equations:

\begin{equation*}
F_1(Z):=\sum_{j=1}^n\lambda^1_{_j}|z_j|^2=0,
\end{equation*}
\begin{equation}\label{E2}
\vdots
\end{equation}
\begin{equation*}
F_m(Z):=\sum_{j=1}^n\lambda^m_{_j}|z_j|^2=0,
\end{equation*}
\begin{equation}\label{E3}
\rho(Z):=\sum_{j=1}^n|z_j|^2=1,
\end{equation}
\noi where $Z=\big(z_1,\dots,z_n\big)$.\\

\noi The weak hyperbolicity condition implies that if $Z=\big(z_1,\dots,z_n\big)$ is a point on the manifold defined by system \eqref{E2} and equation \eqref{E3},  then at least $2m+1$ coordinates of $Z$ are diffe\-rent from zero,  which in turn implies that the rank of the jacobian matrix of the system  \eqref{E2},  is maximal.\\

\noi It has been shown in \cite{Meer} that the intersection of the hypersurfaces defined by this quadratic equations is a manifold of real dimension $2n-2m-1$  that we denote by $\MLamU$. The manifold $\MLamU$ is called \emph{moment-angle manifold}.\\

\noi We denote by $\MLam$ the manifold in $\C^n-\big\{(0,\dots,0)\big\}$ obtained by  system  \eqref{E2}. $\MLam$ is indeed a manifold since it is shown in \cite{Meer} that if the admi\-ssibility condition is satisfied,  the only singularity of the variety given by \eqref{E2} is the origin. 

\begin{rmk}\label{R1}
The moment-angle manifold $\MLamU$admits a canonical action of $\big(\s^1\big)^n$ :
$$
\varphi_{_{\left(u_1,\dots,u_n\right)}}\big(z_1,\dots z_n\big)=\big(u_1z_1\dots,u_n z_n\big),
$$
where $|u_i|=1$ for all $i$. There exists a subtorus $\T^{2m+1}$ which acts locally-free.\\

\noi The orbit space of this action is a full simple polytope of dimension $n-2m-1$. For these facts see \cite{Meer}.
\noi \end{rmk}

\noi Consider, for $a_j$ a positive integer with $j\in\{1,\dots,n\}$,  the following real $1$-form on $\MLamU$:
$$
\alpha:=i\sum_{j=1}^n\Big[a_j\big(z_jd\bar{z}_j-\bar{z}_jdz_j\big)\Big].
$$

\begin{lema}\label{L2}
  For all $Z=\big(z_1,\dots,z_n\big)\in\MLamU$ the $1$-form $\alpha$ restricted to the tangent space $\tange_Z\left(\MLamU\right)$ is a nontrivial form.
\end{lema}

\noi {\bf Remark.} We will consider  $a_j=1$ for all $j\in \{1,\dots,n\}$ since the other cases are completely analogous.

\begin{proof}

\noi For $Z=\big(z_1,\ldots,z_n\big)\in\MLamU$, the linear function 

$$
\alpha_Z:\tange_Z\left(\MLamU\right)\to\R
$$ 
\noi is trivial if and only if there exist $T_1,\dots,T_m\in\C$ and $\mu\in\R$ such that: 

\begin{dmath}\label{E4}
 i\sum_{j=1}^n \big(z_jd\bar{z}_j-\bar{z}_jdz_j\big)
 =\sum_{k=1}^m\left[T_k \sum_{j=1}^n\left[\lambda^k_{_j}\big(\bar{z}_jdz_j+z_jd\bar{z}_j\big)\right]\right]
+\sum_{k=1}^m\left[\overline{T}_k\sum_{j=1}^n\left[\bar{\lambda}^k_{_j}\big(z_jd\bar{z}_j+\bar{z}_jdz_j\big)\right]\right]
+\mu\sum_{j=1}^n \big(z_jd\bar{z}_j+\bar{z}_jdz_j\big). 
\end{dmath}

\noi This is because the set of forms in the right-hand side of this equation is a $(2m+1)$-dimensional real vector space and any form of this type vanishes on $\tange_Z\left(\MLamU\right)$. The set of forms in  $\tange_Z\big(\C^n\big)$ that vanish on $\tange_Z\left(\MLamU\right)$ is also $(2m+1)$-dimensional.  Indeed: for all $k\in\{1,\dots,m\}$, the real differential forms
$$\omega^k_{_1}(Z)=d\big(\Re\left(F_k\right)\big)(Z);\;\;\omega^k_{_2}(Z)=d\big(\Im\left(F_k\right)\big)(Z),\;\;\omega^k_{_3}(Z)=d\rho(Z)
$$ 
are, by hypothesis, linearly independent since the hypersurfaces defined by 
$$
\Re\big(F_k\big)(Z)=0,\quad\Im\big(F_k\big)(Z)=0\quad\text{and}\quad\rho(Z)=1,\quad \forall\; k=1,\dots,m
$$
\noi intersect transversally.\\

\noi Let $v\in \tange_Z\left(\MLamU\right)$, we have that $\omega^k_{_1}(Z)=\omega^k_{_2}(Z)=\omega_{_3}(Z)=0$, for all $k\in\{1,\dots,m\}$,  then the set of real 1-forms vanishing identically on  $\tange_Z\left(\MLamU\right)$ is of dimension $2m+1$. Therefore any 1-form vanishing identically in $\tange_Z\left(\MLamU\right)$ is of the form
$$
\sum_{k=1}^m\Big[\left(a_k+ib_k\right)dF_k\Big]+\sum_{k=1}^m\Big[\left(a_k-ib_k\right)d\overline{F_k}\Big]+\mu d\rho.
$$ 

\noi Therefore the set of linear forms defined  by the right hand side is e\-xac\-tly the set of forms that vanish identically  on $\tange_Z\left(\MLamU \right)$. \\

\noi Comparing coefficients in equation \eqref{E4} we conclude that 
the origin $Z=\left(0,\dots,0\right)$ is the unique point which satisfies the last equations. But the origin is not in $\MLamU$. We conclude that $\alpha$ is nontrivial on $\MLamU$. 
\end{proof}

\noi We will denote by $\K_\alpha(Z)$ the kernel of $\alpha$ at the point $Z\in\MLamU$.

\begin{prop}\label{P1}
Let $n>3$, $m\geq1$ and $n>2m$.  The kernel of $d\alpha$ for all  $Z\in \tange_Z\left(\MLamU\right)$, is a $(2m+1)$-dimensional real subspace of  $\tange\left(\MLamU\right)$.
\end{prop}

\noi {\it Proof:} Points $v\in\C^{n}$ in the kernel of $d\alpha$ restricted to $\tange\left(\MLamU\right)$ satisfy:

\begin{equation}\label{E5}
\begin{cases}
dF_1(v)&=0,\\
\vdots\\
dF_m(v)&=0,\\
d\rho(v)&=0,\\
\iota_v d\alpha&=\sum_{k=1}^mT_kdF_k+\sum_{k=1}^m\overline{T}_kd\overline{F_k}+\mu d\rho,
\end{cases}
\end{equation}

\noi Let $Z=\big(z_1,\dots,z_n\big)\in\MLamU$ and $v=\big(v_1,\dots,v_n\big)\in\tange_Z\left(\MLamU\right)$. The system \eqref{E5} corresponds  to the following equations:

\begin{equation*}
 \sum_{j=1}^n \left[\lambda^1_{_j}\biggr(z_jd\bar{z}_j(v)+\bar{z}_jdz_j(v)\biggl)\right]=0,
\end{equation*}
\begin{equation*}
\vdots
\end{equation*}
\begin{equation*}
 \sum_{j=1}^n \left[\lambda^m_{_j}\biggr(z_jd\bar{z}_j(v)+\bar{z}_jdz_j(v)\biggl)\right]=0,
\end{equation*}

\begin{equation*}
\sum_{j=1}^n\biggr(\bar{z}_jdz_j(v)+z_jd\bar{z}_j(v)\biggl)=0,
\end{equation*}

\begin{dmath}\label{E6}
 i\sum_{j=1}^n\biggr(dz_j(v)d\bar{z}_j-d\bar{z}_j(v)dz_j\biggl)
=\sum_{k=1}^m\left[T_k \sum_{j=1}^n\left[\lambda^k_{_j}\big(\bar{z}_jdz_j+z_jd\bar{z}_j\big)\right]\right]
+\sum_{k=1}^m\left[\overline{T}_k\sum_{j=1}^n\left[\bar{\lambda}^k_{_j}\big(z_jd\bar{z}_j+\bar{z}_jdz_j\big)\right]\right]
+\mu\sum_{j=1}^n \left(z_jd\bar{z}_j+\bar{z}_jdz_j\right). 
\end{dmath}

\noi Comparing  coefficients in  equation \eqref{E6} we have:
$$
idz_j(v)=\left[2\Re\left(\sum_{k=1}^m T_k\lambda^k_{_j}\right)+\mu\right]z_j,\quad\quad j\in\{1,\dots,n\}.
$$

\noi  Then,  for $Z=\big(z_1,\dots,z_n\big)\in\MLamU$ vectors $v\in\tange_Z\left(\MLamU\right)$ satisfying $\iota_vd\alpha=0$ are of the form:

$$
v\big(T_1,\dots,T_m,\mu;Z\big)=
$$
$$-i\left(\left(2\Re\left(\sum_{k=1}^mT_k\lambda^k_{_1}\right)+\mu\right)z_1,\dots,\left(2\Re\left(\sum_{k=1}^mT_k\lambda^k_{_n}\right)+\mu\right)z_n\right),
$$

\noi where $T_1,\dots,T_m\in\C$  and $\mu\in\R$. \\

\noi The vector $v\big(T_1,\dots,T_m,\mu;Z\big)$ is automatically tangent to the sphere  $\s^{2n-2m+1}$ and satisfies the condition $dF_k(v)=0$  for all $k\in\{1,\dots,m\}$.

\begin{lema}\label{L3}
The $\R$-linear map $\phi_Z:\C^m\times\R\cong\R^{2m+1}\to\C^n\cong\R^{2n}$ defined by 
$$
\phi_Z\left(T_1,\dots,T_m,\mu;Z\right)= v(T_1,\dots,T_m,\mu;Z),
$$
\noi is injective.
\end{lema}

\begin{proof}
Since $\Lam$ is admissible the $(2m+1)\times(2m+1)$-real matrix whose rows are $\left(\lambda^1_{_j},\dots,\lambda^m_{_j}\right)$ for $j\in\{1,\dots,n\}$ and $(1,\dots,1)$ is of maximal rank (see \cite[Lemma I.1]{Meer}). Therefore the kernel of $\phi_Z$ is trivial. 
\end{proof}
\noi Then, $\ker(d\alpha_Z)$ has real dimension $2m+1$ and we finish the proof of the proposition.\\

\noi Let $v\left(T_1,\dots,T_m,\mu;Z\right)$ be a vector in $\ker\big(d\alpha_Z\big)$.  We want to know  when this vector is in the kernel of $\alpha_Z$; i.e., when $v\big(T_1,\dots,T_m,\mu;Z\big)$ satisfies the following equation:
 \small{\begin{eqnarray*}  
\sum_{j=1}^n
\left(
\left(2\Re\left(\sum_{k=1}^mT_k\lambda^k_{_j}\right)+\mu\right)|z_j|^2+
\left(2\Re\left(\sum_{k=1}^mT_k\lambda^k_{_j}\right)+\mu\right)|z_j|^2
\right)=0.
\end{eqnarray*}
}
 
\noi Then
$$
 \mu+2\sum_{j=1}^n\Re\left(\sum_{k=1}^mT_k\lambda^k_{_j}\right)|z_j|^2=0. 
$$
 
\noi It follows that $\mu=0$.\\

\noi Thus 
$$\big[\K_\alpha(Z)\cap\ker d\alpha_Z\big]=\biggr\{v\big(T_1,\dots,T_m,0;Z\big)\;\Big|\;T_k\in\C\;\;\text{for all}\; \;k=1,\dots,m\biggl\},$$

\noi where  vectors $v\big(T_1,\dots,T_m,0;Z\big)$ are of the form:
\begin{equation}\label{E7}
-i\left(2\Re\left(\sum_{k=1}^mT_k\lambda^k_{_1}\right)z_1,\dots,2\Re\left(\sum_{k=1}^mT_k\lambda^k_{_n}\right)z_n\right).
\end{equation}

\noi Therefore: 
$$
\dim\Big[\K_\alpha(Z)\cap\ker(d\alpha_Z)\Big]=2m.$$ 

\begin{rmk}\label{RD}
\noi By a general version of  Darboux's  theorem (see \cite{St}, Chapter 3, Theorem 6.2),  for each $Z\in\MLamU$ there exist local coordinates
$$
\big(x_1,y_1,\dots,x_m,y_m,z,u_1,v_1,\dots,u_{n-2m-1},v_{n-2m-1}\big)\in\R^{(2n-2m-1)}, 
$$
in an open neighborhood $\mathcal U$ of $Z$ such that the $1$-form $\alpha$ can be written as
$$
\alpha=dz+u_1dv_1+\dots,u_{n-2m-1}dv_{n-2m-1}.
$$

\noi Furthermore $\Big[\K_\alpha(Z')\cap\ker(d\alpha_{Z'})\Big]$ where $Z'\in\U$, in this neighborhood is the distribution which is tangent to the foliation in $\mathcal U$ whose leaves are $z=\rm{constant}$, $u_i=\rm{constant}$, $v_i=\rm{constant}$ ($i=1,\dots, n-2m-1$).
\end{rmk}

\begin{rmk}\label{R4}
\noi The leaves of the foliation $\mathcal{F}_\alpha$ are the orbits of the locally free action of $\R^{2m}$ on $\MLamU$  given by 
$$
\psi_{_{\left(T_1,\dots,T_m\right)}}\big(z_1,\dots,z_n\big)=\left(e^{-2i\Re\left(\sum_{k=1}^mT_k\lambda^k_{_1}\right)}z_1,\dots,e^{-2i\Re\left(\sum_{k=1}^mT_k\lambda^k_{_n}\right)}z_n\right).
$$
\noi These leaves are tangent to vectors of the form \eqref{E7}.
\end{rmk}

\begin{rmk} \label{R3} It follows that the $2m$-dimensional distribution $\Big[\K_\alpha(Z)\cap\ker(d\alpha_Z)\Big]$ is tangent to a foliation $\mathcal{F}_{\alpha}$. One has that $d\alpha$ is a totally  degenerate closed $2$-form,  when restricted to the foliation  $\mathcal{F}_{\alpha}$. However, since the leaves of the foliation $\mathcal{F}_{\alpha}$ are given as the orbits of a locally free action of $\R^{2m}$ it follows immediately that the foliation $\mathcal{F}_{\alpha}$ is a symplectic foliation with leaves of dimension $2m$.
\end{rmk} 

\noi  From the previous remark it follows:

\begin{prop}\label{P2}
\noi Let $m\geq1$, $n>3$ such that $n>2m$ and $\Lam$ be an admissible configuration. Then the $(2n-2m-1)$-moment-angle manifold $\MLamU$ associated to $\Lam$ has a regular Poisson structure,  whose associated symplectic foliation $\mathcal{F}=\mathcal{F}_{(\Lam,m,n)}$ is of dimension $2m$.
\end{prop}

\section{Moment-angle manifolds of mixed type.}\label{S3}

\noi We  define a generalization of the moment-angle manifolds, already defined,  which are obtained by adding the sum of squares of new complex variables and we call them \emph{moment-angle manifolds of mixed type}. \\

\noi  First we consider the case when $m=1$, $n>3$ and $s\geq1$.\\

\noi Let $X=\big(w_1,\dots,w_s,z_1,\dots,z_n\big)\in\C^{n+s}$.  Let $\Lam=\big(\lambda_1,\dots,\lambda_n\big)$ be an admi\-ssi\-ble configuration,  with $\lambda_j\in\C$ for all $j\in\{1,\dots,n\}$. 

 \begin{defi}\label{D2}
 The manifold $\Mg$ obtained as the intersection of the hypersurfaces defined by the quadratic equations:
\begin{eqnarray}\label{E8}
F_s(X)&:=&\sum_{r=1}^s w^2_r+\sum_{j=1}^n\lambda_j|z_j|^2=0,
\\
\label{E9}
\rho_s(X)&:=&\sum_{r=1}^s|w_r|^2+\sum_{j=1}^n|z_j|^2=1,
\end{eqnarray}
is called \emph{moment-angle manifold of mixed type}.
\end{defi}

\begin{defi}
We denote by $\Mgo$ the manifold consisting of  all points 

\noi $\big(w_1,\dots,w_s,z_1,\dots,z_n\big)$ in $\C^{n+s}-\big\{(0,\dots,0)\big\}$ satisfying equation \eqref{E8}.
\end{defi}

\noi $\Mgo$ is indeed a manifold since the only singularity of the variety determined by \eqref{E8} is the origin.\\

\noi  In \cite{GL} it was shown that $\Mg$ is indeed a manifold. This is true since $\Mgo$ is transverse to the unit sphere.\\

\noi If one deforms continuously an admissible configuration $\Lam_0$ through a family  $\{\Lam_t\}_{t\in(0,1)}$, in such a way that $\Lam_t$  is admissible for all $t\in(0,1)$, then the corres\-ponding manifolds are isotopic submanifolds of the unit sphere. \\

\noi Therefore, the manifolds are isotopic to manifolds which correspond to configurations given by the vertices of regular polygons with an odd number $2\ell+1$ of vertices with a positive weight attached at each vertex, since every admissible configuration can be continuously deformed to such a configuration (see \cite{Meer} and \cite{BM}).\\

\noi  In other words, the manifolds are determined, up to diffeomorphism,  by cyclic sequences of integers $n_1,\dots, n_{2\ell+1}$, with $n_1+n_2+\cdots+n_{2\ell+1}=n$. Hence, for a fixed $n>3$ the number $N(n)$ of manifolds $\Mg$, up to diffeomorphisms, is finite. Furthermore $
N(n)\to\infty,  \,\,\,\,\text{as}\,\,\,t\to \infty$.\\

\noi The  manifolds $\Mg$ are completely understood as we see in the following theorem:

\begin{theo}\cite{GL}\label{T1}
Let $m=1$, $n>3$ and $s\geq1$. Let $\Lam=\big(\lambda_1\dots,\lambda_n\big)$ be an admi\-ssible  configuration, with $\lambda_j\in\C$. The moment-angle manifold of mixed type $\Mg$ has dimension $2n+2s-3$ and it is diffeomorphic to:
$$\underset{j=1}{\overset{2\ell+1}{\sharp}}\left(\s^{2d_j+s-1}\times\;\s^{2n-2d_j+s-2}\right),\quad
 \text{\it where}\quad d_j=n_j+\dots+n_{j+\ell-1}.
$$
\end{theo}

\begin{rmk}\label{R5}
\noi When $m=1$ and $s\geq1$ is arbitrary, V. G\'omez  showed in \cite{GL} that the moment-angle manifold of mixed type $\Mg$ is a two branched covering of the sphere $\s^{2(n+s)-3}$ branched along $\MgU$. This permits to describe  $\Mg$ as an iterated sequence of double branched coverings.
\end{rmk}

\noi Consider, for every set of $n+s$ positive numbers $a_1,\dots,a_s,b_1,\dots,b_n$,  the real $1$-form  on $\Mg$ :

$$
\alpha:= i\left[\sum_{r=1}^s\Big[a_r\big(w_rd\bar{w}_r-\bar{w}_rdw_r\big)\Big]+\sum_{j=1}^n\Big[b_j\big(z_jd\bar{z}_j-\bar{z}_jdz_j\big)\Big]\right].
 $$

 \begin{lema}\cite[Lemma 1]{BLV}\label{L4}
 Let $m=1$, $n>3$ and $s\geq1$. Let $\Lam$ be an admissible configuration in $\C$. Then, for all $X=\big(w_1,\dots,w_s,z_1,\dots,z_n\big)\in\Mg$ the $1$-form $\alpha$ restricted to the tangent space $\tange_X\left(\Mg\right)$ is a nontrivial form.
\end{lema}

\noi We will denote by $\K_\alpha(X)$ the kernel of $\alpha$ at the point $X\in\Mg$.\\

\noi The restriction of $d\alpha$ to each tangent space $\tange_X\left(\Mg\right)$ is skew-symmetric and since the real dimension of $\Mg$ is odd, $\ker(d\alpha)$ should be of dimension at least one.

\begin{prop}\label{P3}
Let $m=1$, $n>3$, $s\geq1$. Then:
\begin{enumerate}
\item If $X=(w_1,\dots,w_s, z_1\dots,z_n)\in\Mg$ is such that $w_1^2+\dots+w_s^2\neq0$,  then the kernel of $d\alpha$ at the point $X\in\tange_X\left(\Mg\right)$ is a real $1$-dimensional subspace of $\tange\left(\Mg\right)$.  

\item If $X$ is a point of $\Mg$ such that $w_1^2+\dots+w_s^2=0$,  then the kernel of $d\alpha$ at the point $X\in\tange_X\left(\Mg\right)$ is a real $3$-dimensional subspace of $\tange\left(\Mg\right)$.
\end{enumerate}

\end{prop}
\begin{proof}
A point $v\in\C^{n+s}$ which belongs to $\ker(d\alpha)$ restricted to $\tange\left(\Mg\right)$,  satisfies the following equations:
\begin{equation}\label{E13}
\begin{cases}
dF_s(v)&=0,\\
d\rho_s(v)&=0,\\
\iota_v d\alpha&=TdF_s+\overline{T}d\overline{F_s}+\mu d\rho_s.
\end{cases}
\end{equation}

\noi Let $X=\left(w_1,\dots,w_s,z_1,\dots,z_n\right)\in\Mg$ and $v=\left(u_1,\dots,u_s,v_1,\dots,v_n\right)\in\tange_X\left(\Mg\right)$. The system  \eqref{E13} corresponds to:

\begin{equation*}
\sum_{r=1}^s 2w_rdw_r(v)+\sum_{j=1}^n \Big[\lambda_j\big(z_jd\bar{z}_j(v)+\bar{z}_jdz_j(v)\big)\Big]=0,
\end{equation*}

\begin{equation*}
\sum_{r=1}^s \Big(\bar{w}_rdw_r(v)+w_rd\bar{w}_r(v)\Big)+\sum_{j=1}^n\Big(\bar{z}_jdz_j(v)+z_jd\bar{z}_j(v)\Big)=0,
\end{equation*}

\begin{dmath}\label{E14}
\nonumber i\left[\sum_{r=1}^s \Big(dw_r(v)d\bar{w}_r-d\bar{w}_r(v)dw_r\Big)+\sum_{j=1}^n\Big(dz_j(v)d\bar{z}_j-d\bar{z}_j(v)dz_j\Big)\right]\\
\nonumber =T\left[\sum_{r=1}^s2w_rdw_r+\sum_{j=1}^n\Big[\lambda_j\big(\bar{z}_jdz_j+z_jd\bar{z}_j\big)\Big]\right] \\
+\overline{T}\left[\sum_{r=1}^s2\bar{w}_rd\bar{w}_r+\sum_{j=1}^n\Big[\bar{\lambda}_j\big(z_jd\bar{z}_j-\bar{z}_jdz_j\big)\Big]\right] \\
\nonumber +\mu \left[\sum_{r=1}^s \Big(w_rd\bar{w}_r+\bar{w}_rdw_r\Big)+\sum_{j=1}^n\Big(z_jd\bar{z}_j+\bar{z}_jdz_j\Big)\right].\end{dmath}

\noi Comparing coefficients in  equation \eqref{E14} we have:
$$
idw_r(v)=2\overline{T}\bar{w}_r+\mu w_r, \quad \quad r\in\{1,\dots,s\};
$$
$$
idz_j(v)=\Big(2\Re(T\lambda_j)+\mu\Big)z_j,\quad\quad j\in\{1,\dots,n\}.
$$

\noi  Then, vectors $v\in\tange_X\left(\Mg\right)$ where $\iota_vd\alpha=0$ i.e, vectors in the kernel of $d\alpha$ at 
$X=\big(w_1,\dots,w_s, z_1,\dots,z_n\big)$ are of the form:

$$
v\big(T,\mu;X\big)=
$$
$$
-i\Big(2\overline{T}\bar{w}_1+\mu w_1,\dots,2\overline{T}\bar{w}_s+\mu w_s,\big(2\Re(T\lambda_1)+\mu\big)z_1,\dots,\big(2\Re(T\lambda_n\big)+\mu)z_n\Big),
$$
\noi where $T\in\C$ and $\mu\in\R$.\\

\noi The vector $v\big(T,\mu;X\big)$ has to be tangent to the sphere  $\s^{2n+2s-1}$, and therefore it must satisfy this condition:
$$
\Re\left(-i\sum_{r=1}^s\Big(2\overline{T}\bar{w}^2_r+\mu|w_r|^2\Big)-i\sum_{j=1}^n\Big(2\Re(T\lambda_j)+\mu\Big)|z_j|^2\right)=0.
$$
\noi In other words:
$$
\Re\left(-i\sum_{r=1}^s{\overline{T}}\bar{w}^2_r\right)=0.
$$

\noi  If $w_1^2+\dots+w_s^2\neq0$ , $T$ depends only on a real parameter because it must be of the form 
$$T=t\sum_{r=1}^s\bar{w}^2_r,\;\;\; \text{for some} \,\,   t\in\R.$$

\noi If $w_1^2+\dots+w_s^2=0$, $T$ can be any complex number.\\

\noi The condition $dF_s(v)=0$ and $v\big(T,\mu;X\big)\in\s^{2n+2s-1}$ implies:

$$
2\overline{T}\sum_{r=1}^s|w_r|^2+\mu\sum_{r=1}^sw^2_r=0.
$$
If $w_1^2+\dots+w_s^2\neq0$ we have that $\mu$ depends on $t$ as follows:
$$
\mu=-2t\sum_{r=1}^s|w_r|^2.
$$
\noi If $w_1^2+\dots+w_s^2=0$ then $\mu$ can be an arbitrary real number.\\

\noi Therefore  at a point where  $w_1^2+\dots+w_s^2\neq0$,  vectors in the kernel of $d\alpha_X$ depend only on a real parameter $t$ so the kernel of $d\alpha_X$ has dimension one.\\

\noi At a point where $w_1^2+\dots+w_s^2=0$ ,  vectors in the kernel of $d\alpha_X$ depend on the complex number $T$ and a real number $\mu$ so the kernel of $d\alpha_X$ has dimension three, in other words the kernel of $d\alpha_X$  consists of all vectors of the form:
$$
v\big(T,\mu;X\big)=
$$
$$
-i\Big(2\overline{T}\bar{w}_1+\mu w_1,\dots,2\overline{T}\bar{w}_s+\mu w_s,\big(2\Re(T\lambda_1)+\mu\big)z_1,\dots,\big(2\Re(T\lambda_n)+\mu\big)z_n\Big)
$$

\end{proof}

\noi Let $W_s$ be the set of points $\big(w_1,\dots,w_s,z_1,\dots,z_n\big)$ of $\Mg$ such that $w_1^2+\dots+w_s^2=0$. Then
$W_s$ is a real analytic variety of $\Mg$ of real codimension 2s. The variety $W_s$ is singular if $s>1$ and when $s=1$ is a moment-angle manifold of dimension $2n-3$.\\

%\noi Let now $s=1$.\\
%
\noi Let $v\big(T,\mu;X\big)$ be a vector in $\ker(d\alpha_X)$. We want to know when $v\big(T,\mu;X\big)$ is in $\K_\alpha(X)$; in other words, when $v\big(T,\mu;X\big)$  satisfies the following equation:
$$
i\left[\sum_{r=1}^s\Big(w_rd\bar{w}_r(v)-\bar{w}_rdw_r(v)\Big)+\sum_{j=1}^n\Big(z_jd\bar{z}_j(v)-\bar{z}_jdz_j(v)\Big)\right]=0.
$$
\noi Then
$$
2\sum_{r=1}^s\Re\left(Tw^2_r\right)+2\sum_{j=1}^n\Re\left(T\lambda_j\right)|z_j|^2+\mu=0,
$$
\noi it follows that $\mu=0$.\\

\noi  Hence, if $w^2_1+\dots+w^2_s\not= 0$ we have $\mu=-2t\sum_{r=1}^m|w_r|^2=0$, we have  that $t$ must be equal to zero.\\

\noi Therefore in the set of points $X=\big(w_1,\dots,w_s,z_1,\dots,z_n\big)\in\Mg$ such that $w^2_1+\dots+w^2_s\not=0$,  the dimension of the subspace $\Big[\K_\alpha(X)\cap\ker(d\alpha_X)\Big]$ is zero. It follows that  the $1$-form $\alpha$ is a contact form in $\Mg-W_s$. In other words,
$$
Rank\Big(d\alpha_X|_{\K_\alpha(X)}\Big)=2(n+s-2),
$$
when $w^2_1+\dots+w^2_s\not=0$.\\

%\newpage

\noi On the other hand,  in the set of points $X=\big(0,\dots,0, z_1,\dots,z_n\big)\in\Mg$ and assuming that $\mu=0$, vectors of the form
$$
v\big(T,0;X\big)=-i\Big(0,\dots,0,\Re\left(T\lambda_1\right)z_1,\dots,2\Re\left(T\lambda_n\right)z_n\Big), \quad \quad T\in\C,
$$
are in $\Big[\K_\alpha(X)\cap\ker(d\alpha_X)\Big]$.\\

\noi Hence $\Big[\K_\alpha(X)\cap\ker(d\alpha_X)\Big]$ is a two dimensional space parametrized by $T\in\C$.

\begin{defi}
We will denote this $2$-dimensional vector space at the point $X\in{W_s}$ by $\Pi_s(X)$.
\end{defi}

\noi Let now $m> 1$, $n>3$ such that $n>2m$. \\

\noi Let $\Lam=\big(\lam_1,\dots,\lam_n\big)$  be an $n$-tuple of vectors in  $\C^m$, with $\lam_j=\big(\lambda^1_{_j},\dots,\lambda^m_{_j}\big)$ for all  $j\in\{1,\dots,n\}$ and let $w_k\in\C$ for all $k\in\{1,\dots,m\}$. \\

\noi Let consider the following  system of quadratic equations:
\begin{equation}\label{E15}
\begin{cases}
\pmb{F_1}(X):=w^2_1+\sum_{j=1}^n\lambda^1_{_j}|z_j|^2=0,\\
\vdots\\
\pmb{F_m}(X):=w^2_m+\sum_{j=1}^n\lambda^m_{_j}|z_j|^2=0.
\end{cases}
\end{equation}

\noi We use the notation $F_k(Z)=\sum_{j=1}^n\lambda^k_{_j}|z_j|^2$ and $G_k(X)=w^2_k+\sum_{j=1}^n\lambda^k_{_j}|z_j|^2$.

\begin{prop}\label{P4}
The set of zeros of system \eqref{E15} is a regular manifold outside of the origin if and only if, for all collection $K\subset\{1,\dots,m\}$ the set of zeros of the system
\begin{equation}\label{E16}
F_k(Z)=0,\quad\quad k\in K,
\end{equation}
is a regular manifold outside the origin.
\end{prop}

\begin{proof}
$\Rightarrow)$ Let suppose that the set of zeros of \eqref{E15} is non a regular manifold outside the origin. Let $(W,Z)\not=(0,0)$ a singular solution of \eqref{E15}, where $W=\big(w_1,\dots,w_m\big)$ and $Z=\big(z_1,\dots,z_n\big)$.\\

\noi Let $K$ be the set of $k\in\{1,\dots,m\}$ such that $w_k=0$. We can assume that $K=\{r+1,\dots,m\}$, with $1\leq r\leq m$. Then the jacobian matrix of system \eqref{E15} is the following:
$$
\left[\begin{array}{cc}\mathcal{W}_{2r\times2m}& \mathcal{A}_{2r\times2n} \\\mathbf{0}_{2(m-r)\times2m} &\mathcal{B}_{2(m-r)\times2n} \end{array}\right],
$$
\noi where $\mathbf{0}_{2(m-r)\times2m}$ is the ${2(m-r)\times2m}$-zero matrix and $\mathcal{W}_{2r\times2m}$ is a ${2r\times2m}$-diagonal matrix of rank $2r$, where the elements in the diagonal are $2\times2$-matrices. $\mathcal{A}_{2r\times2n}$ is the ${2r\times2n}$-jacobian matrix of the system  given by equations $F_k(Z)$ for $k\in\{1,\dots,r\}$ and $\mathcal{B}_{2(m-r)\times2n} $ is the ${2(m-r)\times2n}$-jacobian matrix of the system \eqref{E16} with $k\in K=\{r+1,\dots,m\}$, then $Rank(\mathcal{B})<2(m-r)$ in the point $Z=\left(z_1,\dots,z_n\right)$.\\

\noi  Observe that if $Z=(0,\dots,0)$ then $w_k$ will be $0$ for all $k\in\{1,\dots,m\}$.\\

\noi Then the set of zeros of the system \eqref{E16} is non a regular manifold outside the origin.\\

\noi $\Leftarrow)$ Suppose that $\mathcal{B}_{2(m-r)\times2n} $ is singular in the point $Z\not=0$ which is a solution of the system \eqref{E16} and let $w_k=0$ for $k\in K=\{r+1,\dots,m\}$. Observe that $Z$ is a solution of the equations $F_k(Z)=0$ for $k\in\{1,\dots,r\}$. We can find $w_k$ for $k\in\{1,\dots,r\}$ such that $G_k\big((W,Z)\big)=0$ for $k\in\{1,\dots,r\}$.\\

\noi  Then, the point $(W,Z)$ is a solution of \eqref{E15} and the rank of the jacobian matrix of \eqref{E16} is less than $2m$.
\end{proof}

\begin{prop}\label{P5}
Let $m\geq1$, $n>3$ and $n>2m$. Let $\Lam=\big(\lam_1,\dots,\lam_n\big)$ be an $n$-tuple of vectors in $\C^m$ and let $w_k\in\C$ for all $k\in\{1,\dots,m\}$. The intersection of the hypersurfaces defined by the following system of equations:

\begin{equation*}
\pmb{F_1}(X):=w^2_1+\sum_{j=1}^n\lambda^1_{_j}|z_j|^2=0,
\end{equation*}
\begin{equation*}
\vdots
\end{equation*}
\begin{equation*}
\pmb{F_m}(X):=w^2_m+\sum_{j=1}^n\lambda^m_{_j}|z_j|^2=0,
\end{equation*}
\begin{equation*}
\pmb{\rho}(X):=\sum_{k=1}^m|w_k|^2+\sum_{j=1}^n|z_j|^2=1,
\end{equation*}

\noi defines a manifold $\MG$ of real dimension $2n-1$ if 
\begin{enumerate}
\item $\Lam$ is an admissible configuration in $\C^m$,
\item for all collection of $k\in\{1,\dots,m\}$ such that $w_k=0$, for instance $K=\{r_1,\dots,r_\ell\}$ with $1\leq r_{\ell}\leq m$ and $\ell\in\{1,\dots,m\}$, we have that the $n$-tuple $\Lam'=\big(\lam'_1,\dots,\lam'_n\big)$  of vectors $\lam'_j=\big(\lambda^{r_1}_{_j},\dots,\lambda^{r_\ell}_{_j}\big)\in\C^{m-\ell}$ is an admissible configuration.
\end{enumerate}
\end{prop}

\begin{proof}
 \noi The proof follows from proposition \ref{P4} since the manifolds determined by equation \eqref{E15} are transverse to the unit sphere since they are homogeneous of degree 2.
 \end{proof}

\begin{obs}\label{O1}
$\MG$ is not empty since $\Lam$ is an admissible configuration in $\C^m$ and therefore $\MG$ contains the submanifold of points with $w_k=0$ for all $k\in\{1,\dots,m\}$.
\end{obs}

\begin{defi}\label{D5}
We call the manifold $\MG$ \emph{moment-angle manifold of mixed type} correspon\-ding to the admissible configuration $\Lam=\big(\lam_1,\dots,\lam_n\big)$, with  $\lam_j\in\C^m$.
\end{defi}

\begin{defi}\label{D6}
We denote by $\MGo$ the manifold obtained by system \eqref{E15} in $\C^{n+m}-\big\{(0,\dots,0)\big\}$.
\end{defi}

\noi $\MGo$ is indeed a manifold since the only singularity of the variety given by \eqref{E15} is the origin. If $\big(w_1,\dots,w_m,z_,\dots,z_n\big)$ satisfies \eqref{E15} then for all real number $t$ the point $\big(tw_1,\dots,tw_m,tz_1,\dots,tz_n\big)$ also satisfies \eqref{E15}, therefore if we add the origin we obtain a real cone with vertex the origin.\\

\noi Consider, for every set of $n+m$ positive numbers $a_1,\dots,a_m$, and $b_1,\dots,b_n$,  the real $1$-form  on $\MG$ :

$$
\alpha:= i\left[\sum_{r=1}^m\Big[a_r\big(w_rd\bar{w}_r-\bar{w}_rdw_r\big)\Big]+\sum_{j=1}^n\Big[b_j\big(z_jd\bar{z}_j-\bar{z}_jdz_j\big)\Big]\right].
 $$
 
\begin{lema}\label{L5}
For all $X=\big(w_1,\dots,w_m,z_1,\dots,z_n\big)\in\MG$ the $1$-form $\alpha$ restricted to the tangent space $\tange_X\left(\MG\right)$ is a nontrivial form.
\end{lema}
 
 \begin{proof}
 For $X=\big(w_1,\dots,w_m,z_1,\dots,z_n\big)\in\MG$, the linear function 
 $$
 \alpha_X:\tange_X\left(\MG\right)\to\R
 $$
 is trivial if and only if there exist $T_k\in\C$, $k\in\{1,\dots,m\}$ and $\mu\in\R$ such that:
 
 \begin{dmath}\label{E17}
\nonumber i\left[\sum_{r=1}^m\left(w_rd\bar{w}_r-\bar{w}_rdw_r\right)+\sum_{j=1}^n \left(z_jd\bar{z}_j-\bar{z}_jdz_j\right)\right]\\
\nonumber =\sum_{k=1}^m\left[T_k\left(2w_kdw_k+ \sum_{j=1}^n\left[\lambda^k_{_j}\left(\bar{z}_jdz_j+z_jd\bar{z}_j\right)\right]\right)\right]\\
+\sum_{k=1}^m\left[\overline{T}_k\left(2\bar{w}_kd\bar{w}_k+\sum_{j=1}^n\left[\bar{\lambda}^k_{_j}\left(z_jd\bar{z}_j+\bar{z}_jdz_j\right)\right]\right)\right]\\
\nonumber+\mu\left[\sum_{r=1}^m\left(w_rd\bar{w}_r+\bar{w}_rdw_r\right)+\sum_{j=1}^n \left(z_jd\bar{z}_j+\bar{z}_jdz_j\right)\right]. 
\end{dmath}

 \noi This is because the set of forms in the right-hand side of this equation is a $(2m+1)$-dimensional real vector space and any form of this type vani\-shes on $\tange_X\left(\MG\right)$. The set of forms in  $\tange_X\left(\C^{n+m}\right)$ that vanish on $\tange_X\left(\MG\right)$ is also $(2m+1)$-dimensional.  Indeed: for all $k\in\{1,\dots,m\}$, the real differential forms
$$\omega^k_{_1}(X)=d\big(\Re(G_k)\big)(X);\;\;\omega^k_{_2}(X)=d\big(\Im(G_k)\big)(X),\;\;\omega^k_{_3}(X)=d\pmb{\rho}(X)
$$ 
are, by hypothesis, linearly independent since the hypersurfaces defined by 
$$
\Re\big(G_k\big)(X)=0,\quad\Im\big(G_k\big)(X)=0\quad\text{and}\quad \pmb{\rho}(X)=1,\quad \forall\; k=1,\dots,m
$$
\noi intersect transversally.\\

\noi Let $v\in \tange_X\left(\MG\right)$, we have that $\omega^k_{_1}(X)=\omega^k_{_2}(X)=\omega_{_3}(X)=0$, for all $k\in\{1,\dots,m\}$,  then the set of real 1-forms vanishing identically on  $\tange_X\left(\MG\right)$ is of dimension $2m+1$. Therefore any 1-form vanishing identically in $\tange_X\left(\MG\right)$ is of the form
$$
\sum_{k=1}^m\Big[\left(a_k+ib_k\right)dG_k\Big]+\sum_{k=1}^m\Big[\left(a_k-ib_k\right)d\overline{G_k}\Big]+\mu d\pmb{\rho}.
$$ 

\noi Therefore the set of linear forms defined  by the right hand side is e\-xac\-tly the set of forms that vanish identically  on $\tange_X\left(\MG \right)$. \\

\noi Comparing coefficients in equation \eqref{E17} we have:
\begin{equation}\label{E18}
iw_k=2\overline{T}_k\bar{w}_k+\mu w_k,\quad\quad k\in\{1,\dots,m\},
\end{equation}

\begin{equation}\label{E19}
iz_j=\left(2\Re\left(\sum_{k=1}^mT_k\lambda^k_{_j}\right)+\mu\right)z_j,\quad\quad j\in\{1,\dots,n\}.
\end{equation}

\noi From equation \eqref{E19} we conclude that $z_j=0$ for all $j\in\{1,\dots,n\}$. Then we have points of the form $X=\big(w_1,\dots,w_m,0,\dots,0\big)\in\C^{n+m}$ would be solutions of equations \eqref{E18} and \eqref{E19}, but in this case system \eqref{E15} and equation $\pmb{\rho}$ are of the form:
%\newpage
$$
\pmb{F_1}(X)=w^2_1=0,
$$
$$
\vdots
$$
$$
\pmb{F_m}(X)=w^2_m=0,
$$
$$
\pmb{\rho}(X)=\sum_{k=1}^m|w_k|^2=1.
$$

\noi It follows that $w_k=0$ for all $k\in\{1,\dots,m\}$. Then the unique point satisfying equations \eqref{E18} and \eqref{E19} is the origin $(0,\dots,0)\in\C^{n+m}$,  but the origin is not in $\MG$.\\

\noi We conclude that $\alpha$ is non trivial on $\MG$.
 \end{proof}
 
 \noi We will denote by $\K_\alpha(X)$ the kernel of $\alpha$ at the point $X\in\MG$.
 
 \begin{prop}\label{P6}
 Let $m\geq 1$, $n>3$ and $n>2m$. Then:
 \begin{enumerate}
\item If $X=\big(w_1,\dots,w_m,z_1,\dots,z_n\big)\in \MG$ is such that $w_r\not=0$ for all $r\in\{1,\dots,m\}$, the kernel of $d\alpha$ at the point $X\in\tange_X\left(\MG\right)$ is a real $1$-dimensional subspace of $\tange\left(\MG\right)$.

\item If $X\in\MG$ is such that $\ell$ coordinates $w_r$ are equal to $0$ for $\ell,r\in\{1,\dots,m\}$, the kernel of $d\alpha$ at the point $X\in\tange_X\left(\MG\right)$ is a real $(2\ell+1)$-dimensional subspace of $\tange\left(\MG\right)$.
\end{enumerate}
 \end{prop}
 
 \begin{proof}
 \noi A point $v\in\C^{n+m}$ which belongs to $\ker(d\alpha)$ restricted to $\tange\left(\MG\right)$, satisfies the following equations:
 
 \begin{equation}\label{E20}
\begin{cases}
d\pmb{F_1}(v)&=0,\\
\vdots&\\
d\pmb{F_m}(v)&=0,\\
d\pmb{\rho}(v)&=0,\\
\iota_v d\alpha&=\sum_{k=1}^mT_kd\pmb{F_k}+\sum_{k=1}^m\overline{T}_kd\overline{\pmb{F_k}}+\mu d\pmb{\rho}.
\end{cases}
\end{equation}

\noi Let $X=\big(w_1,\dots,w_m,z_1,\dots,z_n\big)\in\MG$ and $v=\big(u_1,\dots,u_m,v_1,\dots,v_n\big)\in\tange_X\left(\MG\right)$. The system \eqref{E20} corresponds to:
\newpage
\begin{equation*}
2w_1dw_1(v)+\sum_{j=1}^n\lambda^1_{_j}\Big(z_jd\bar{z}_j(v)+\bar{z}_jdz_j(v)\Big)=0,
\end{equation*}
\begin{equation*}
\vdots
\end{equation*}
\begin{equation*}
2w_mdw_m(v)+\sum_{j=1}^n\lambda^m_{_j}\Big(z_jd\bar{z}_j(v)+\bar{z}_jdz_j(v)\Big)=0,
\end{equation*}

\begin{equation*}
\sum_{r=1}^m\Big(\bar{w}_rdw_r(v)+w_rd\bar{w}_r(v)\Big)+\sum_{j=1}^n\Big(\bar{z}_jdz_j(v)+z_jd\bar{z}_j(v)\Big)=0,
\end{equation*}

\begin{dmath}\label{E21}
\nonumber i \left[\sum_{r=1}^m\Big(dw_r(v)d\bar{w}_r-d\bar{w}_r(v)dw_r\Big) + \sum_{j=1}^n\Big(dz_j(v)d\bar{z}_j-d\bar{z}_j(v)dz_j\Big)\right]\\
\nonumber =\sum_{k=1}^m\left[T_k\left(2w_kdw_k+\sum_{j=1}^n\lambda^k_{_j}\Big(z_jd\bar{z}_j+\bar{z}_jdz_j\Big)\right)\right] \\
+\sum_{k=1}^m \left[\overline{T}_k\left(2\bar{w}_kd\bar{w}_k+\sum_{j=1}^n\bar{\lambda}^k_{_j}\Big(\bar{z}_jdz_j+z_jd\bar{z}_j\Big)\right)\right]\\
\nonumber +\mu\left[\sum_{r=1}^m\Big(w_rd\bar{w}_r+\bar{w}_rdw_r\Big)+\sum_{j=1}^n\Big(z_jd\bar{z}_j+\bar{z}_jdz_j\Big)\right].
\end{dmath}

\noi Comparing coefficients in equation \eqref{E21} we have:

$$
idw_r(v)=2\bar{w}_r\overline{T}_r+\mu w_r,\quad\quad r\in\{1,\dots,m\}.
$$

$$
idz_j(v)=\left(2\Re\left(\sum_{k=1}^mT_k\lambda^k_{_j}\right)+\mu\right)z_j,\quad\quad j\in\{1,\dots,n\},
$$

\noi Then, vectors $v\in\tange_X\left(\MG\right)$ where $\iota_vd\alpha=0$; in other words, vectors in the kernel of $d\alpha$ at $X=\big(w_1,\dots,w_m,z_1,\dots,z_n\big)$ are of the form:

\begin{dmath*}
v\left(T_1,\dots,T_m,\mu;X\right)=
-i\left(2\bar{w}_1\overline{T}_1+\mu w_1,\dots,2\bar{w}_m\overline{T}_m+\mu w_m,\\
\left(2\Re\left(\sum_{k=1}^mT_k\lambda^k_{_1}\right)+\mu\right)z_1,
\dots,\left(2\Re\left(\sum_{k=1}^mT_k\lambda^k_{_n}\right)+\mu\right)z_n\right),
\end{dmath*}

\noi where $T_k\in\C$ for all $k\in\{1,\dots,m\}$ and $\mu\in\R$.\\

\noi Consider the condition $d\pmb{F_k}(v)=0$, when $v=v\left(T_1,\dots,T_m,\mu;X\right)$, where $T_k\in\C$ for all $k\in\{1,\dots,m\}$ and $\mu\in\R$:

\begin{dmath*}
2w_k\left[-i\left(2\bar{w}_k\overline{T}_k+\mu w_k\right)+\sum_{j=1}^n\lambda^k_{_j}\left[i\left(2\Re\left(\sum_{k=1}^mT_k\lambda^k_{_j}\right)+\mu\right)|z_j|^2-i\left(2\Re\left(\sum_{k=1}^mT_k\lambda^k_{_j}\right)+\mu\right)|z_j|^2\right]\right]=0.
\end{dmath*}

\noi Then 
$$
-4|w_k|^2\overline{T}_k-2\mu w^2_k=0.
$$
\noi If $w_k\not=0$ for $k\in\{1,\dots,m\}$ it follows that ${T}_k$ depends only on $\mu$:
$$
T_k=T_k(\mu)=-\frac{1}{2}\mu\frac{\bar{w}_k}{w_k}.
$$

\noi The vector $v\left(T_1,\dots,T_m,\mu;X\right)$ has to be tangent to the sphere $\s^{2n+1}$, and therefore it must satisfy the following condition:
{\small$$
\Re\left(-i\sum_{r=1}^m\left(2\bar{w}^2_r\overline{T}_r+\mu|w_r|^2\right)-i\sum_{j=1}^n\left(2\Re\left(\sum_{k=1}^nT_k\lambda^k_{_j}\right)+\mu\right)|z_j|^2\right)=0.
$$}

\noi  In other words:
$$
\Re\left(-2i\sum_{r=1}^m\left(\bar{w}^2_r\overline{T}_r\right)\right)=0.
$$

\noi If we consider the vector $v\big(T_1,\dots,T_m,\mu;X\big)$ with $T_k=\frac{-1}{2}\mu\frac{\bar{w}_k}{w_k}$ when $w_k\not=0$ for all $k\in\{1,\dots,m\}$, we have that it is orthogonal to the sphere $\s^{2n+1}$ since $\mu$ is real:

$$
\Re\left(-2i\sum_{k=1}^m\left(-\frac{1}{2}\bar{w}^2_k\mu\frac{w_k}{\bar{w}_k}\right)\right)=\Re\left(i\sum_{k=1}^m\mu|w_k|^2\right)=0.
$$

\noi If $w_k=0$ for some $k's$, without loss of generality we can suppose that the  first $\ell$ coordinates $w_k$ are $0$, we have:

$$
\Re\left(-2i\left(\sum_{r=1}^\ell 0\overline{T}_r+\sum_{k=\ell+1}^m\mu|w_\ell|^2\right)\right)=0,
$$
\noi  since $\mu$ is real.\\

\noi Observe that even if $w_k=0$ for all $k\in\{1,\dots,m\}$, the vector $v\big(T_1,\dots,T_m,\mu;X\big)$ is orthogonal to the sphere $\s^{2n+1}$.\\

\noi Therefore at a point where $w_k\not=0$ for all $k\in\{1,\dots,m\}$,   vectors in the kernel of $d\alpha$ depend on a real parameter $\mu$, so the kernel of $d\alpha$ has dimension one.\\

\noi Without loss of generality we suppose that the  first $\ell$ coordinates $w_k$ are $0$. We can show, as in  the proof of lemma \ref{L3}, that the $\R$-linear map $\phi_X:\C^{\ell}\times\R\cong\R^{2\ell+1}\to\C^{n+1}\cong\R^{2n+2}$ defined by 
$$
\phi_X\big(T_1,\dots,T_\ell,\mu;X\big)=v\Big(T_1,\dots,T_\ell,T_{\ell+1}(\mu),\dots,T_m(\mu),\mu;X\Big),
$$
is injective.\\

\noi Therefore at a point where $\ell$ coordinates $w_k$ are equal to $0$ with $\ell, k\in\{1,\dots,m\}$,  vectors in the kernel of $d\alpha$ depend on $2\ell$ complex numbers and a real number $\mu$, so the kernel of $d\alpha$ has dimension $2\ell+1$. 
\end{proof}

\noi Let $W_\ell$ be the set of points $\big(w_1,\dots,w_m,z_1,\dots,z_n\big)\in\MG$ such that $\ell$ coor\-dinates $w_k$ are equal to zero for  $\ell, k\in\{1,\dots,m\}$. Then $W_\ell$ is a real analytic subvariety of $\MG$ of real codimension $2\ell$. \\

 \noi Let $v\big(T_1,\dots,T_m,\mu;X\big)$ be a vector in $\ker\big(d\alpha_X\big)$.  We want to know when this vector is in the kernel of $\alpha_X$; i.e., when $v\big(T_1,\dots,T_m,\mu;X\big)$ satisfies the following equation: 
 \small{\begin{eqnarray*}
\sum_{k=1}^m\Big(
\left(2w^2_kT_k+\mu|w_k|^2\right)+\left(2\bar{w}^2_k\overline{T}_k+\mu|w_k|^2\right)
\Big)\\
+\sum_{j=1}^n
\left(
\left(2\Re\left(\sum_{k=1}^mT_k\lambda^k_{_j}\right)+\mu\right)|z_j|^2+
\left(2\Re\left(\sum_{k=1}^mT_k\lambda^k_{_j}\right)+\mu\right)|z_j|^2
\right)=0.
\end{eqnarray*}
}
 
\noi Then
$$
 \mu+2\sum_{k=1}^m\Re\left(w^2_kT_k\right)+2\sum_{j=1}^n\Re\left(\sum_{k=1}^mT_k\lambda^k_{_j}\right)|z_j|^2=0. 
$$
 
\noi Since $T_kw^2_k+\sum_{j=1}^nT_k\lambda^k_{_j}|z_j|^2=0$, it follows that $\mu=0$.\\

\noi Hence, if $\mu=0$ and  $w_k\not=0$ for all $k\in\{1,\dots,m\}$, we have that $T_k=0$ for all $k\in\{1,\dots,m\}$. Therefore, in the set of points $X\in\MG$ such that all coordinates $w_k$ are different from zero, the form $\alpha$ is a contact form. In other words,
$$
Rank\Big(d\alpha_X|_{\K_\alpha(X)}\Big)=2(n-1),
$$
\noi when $w_k\not=0$ for all $k\in\{1,\dots,m\}$.\\

\noi If $\mu=0$ and $\ell$ coordinates $w_k$ are equal to zero with $\ell,k\in\{1,\dots,m\}$, we have that $\Big[\K_\alpha(X)\cap\ker(d\alpha_X)\Big]$ is a $2\ell$-dimensional real subspace parametrized by $\ell$ complex numbers.

\begin{defi}
We will denote this $2\ell$-dimensional vector space at the point $X\in W_\ell$ with $\ell$ coordinates $w_k$ equal to zero by $\Pi_\ell(X)$.
\end{defi}

\subsection{Conductive Confoliations.}\label{SS3}

\noi The following definitions were given by S. J. Altschuler  in \cite{Alt} and generalized by S. J. Altschuler and F. Wu in \cite{AltWu}:

\begin{defi}\label{D8}
If $M^{2\ell+1}$ is a ($2\ell+1$)-manifold, the space of \emph{conductive confoliations}, $Con\left(M^{2\ell+1}\right)$, is defined to be the subset of $\alpha\in \Lam^1\left(M^{2\ell+1}\right)$, the vector space of differentiable 1-forms, such that 
\begin{itemize}
\item $\alpha$ is a positive confoliation: $\ast\Big(\alpha\wedge (d\alpha)^\ell\Big)\geqslant 0$, where $\ast$ denotes the  Hodge operator with respect to some fixed Riemannian metric;
\item every point $p\in{M}^{2\ell+1}$ is accessible from a contact point $q\in{M}^{2\ell+1}$ of $\alpha$: there is a smooth path $\gamma:[0,1]\to{M}^{2\ell+1}$ from $p$ to $q$ with $\gamma'(x)$ in the orthogonal complement of $\ker\Big(\ast\left(\alpha\wedge (d\alpha)^{\ell-1}\right)\Big)$ for  all $x$.
\end{itemize} 
\end{defi}

\noi We have that  $\ast\Big(\alpha\wedge (d\alpha)^{\ell-1}\Big)$ is a $2$-form which we denote by $\tau$. Let us denote the orthogonal complement of $\ker(\tau)$ by $\big[\ker(\tau)\big]^\perp$ and by  $\big[\ker(\tau)\big]^\perp(P)$ this subspace of
$\tange_P\left(M^{2\ell+1}\right)$ at the point $P\in{M}^{2\ell+1}$.

\begin{obs}\label{O2}$\;$
\begin{itemize}
\item At a point in $M$ where $Rank\Big(d\alpha|_{\ker(\alpha)}\Big)=2\ell$, the form $\alpha$ is a contact form on $M$ and $\big[\ker(\tau)\big]^\perp=\ker(\alpha)$  at that point, hence dimension of $\big[\ker(\tau)\big]^\perp$ is equal to $2\ell$.

\item At a point  $P\in{M}$ where $Rank\Big(d\alpha|_{\ker(\alpha)}\Big)=2\ell-2$, the dimension of 
$\big[\ker(\tau)\big]^\perp(P)$ is equal to two. In this case  
$$
\big[\ker(\tau)\big]^\perp(P)=\Big[\ker(\alpha)\cap\ker(d\alpha)\Big](P).
$$
\item At a point  $P\in{M}$ where $Rank\Big(d\alpha|_{\ker(\alpha)}\Big)<2\ell-2$,  $\big[\ker(\tau)\big]^\perp(P)=\{0\}$.
\end{itemize}
\end{obs}

\begin{theo}[Theorem 2.8, \cite{AltWu}]\label{T2}
If $\alpha\in Con(M^{2\ell+1})$ then $\alpha$ is $C^\infty$ close to a contact form.
\end{theo}

\noi These forms are also called \emph{transitive confoliations} by Y. Eliashberg and W. P. Thurston (see \cite{ET}), since we can connect any point of the manifold  to a point where the form $\alpha$ is contact by a Legendrian path of finite length.\\

  \noi Since $W_s$  for $s=1$ is of real codimension two, by the previous results, it follows that it does not disconnect $\Mgu$. Hence:

\begin{prop}\cite[Proposition 1]{BLV}\label{P7}
Let $m=1$, $n>3$ and $s=1$.  Let $\ast$ denote the Hodge opera\-tor for a given Riemannian metric on a moment-angle manifold of mixed type $\Mgu$. Then for the appropriate orientation of $\Mgu$  one has that
\begin{enumerate}
\item for $X=\big(w_1,z_1,\dots,z_n\big)\in\Mgu-W_1$ 
$$
\ast\Big(\alpha\wedge (d\alpha)^{n-1}\Big)(X)>0,
$$
\item for $X\in W_1$
$$
\ast\Big(\alpha\wedge (d\alpha)^{n-1}\Big)(X)=0.
$$
\end{enumerate}
Therefore $\alpha$ is a positive confoliation on $\Mgu$.
\end{prop}

\noi  Let denote by $\tau$ the $2$-form  $\ast\Big(\alpha\wedge(d\alpha)^{n-2}\Big)$.\\

\begin{lema}\cite[Lemma 3]{BLV}\label{L6}
 Let $X\in W_1$. Then, there exists a smooth parametrized curve
$\gamma:[-1,1]\to\Mgu$ such that 
\begin{enumerate}
\item $|\gamma'(x)|\neq0$ for all $x\in[-1,1]$
\item $\gamma(0)=X$ and $\gamma'(0)\in \big[\ker(\tau)\big]^\perp(X)=\Pi_1(X)$
\item If $x\in[-1,1]$ with $x\not=0$, $\gamma(x)\notin{W_1}$  and $\gamma'(x) \in\big[\ker(\tau)\big]^\perp(\gamma(x))=\K_\alpha\big(\gamma(x)\big)$.
\end{enumerate}
\end{lema}

\begin{proof} Let us fix a Riemannian metric $g$. For $P\in W_1$ let $\big[\tange_P(W_1)\big]^\perp$ denote the $2$-dimensional subspace of $\tange\left(\Mgu\right)$, which is orthogonal to $\tange_P(W_1)$ at $P$.\\

\noi Let us first show  that there exists an open neighborhood $\U\subset W_1$ of $X\in W_1$ and a smooth and non-vanishing vector field $\cv:\U\to \tange\left(\Mgu\right)$ defined on $\U$ such that $\cv(P)\in\K_\alpha(P)\cap \big[\tange_P(W_1)\big]^\perp$ for all $P\in\U$.\\

\noi Indeed, Let $\ls(P)=\big[\tange_P(W_1)\big]^\perp\,\cap\,\K_\alpha(P)$. Then $\ls(P)$ has dimension two if $\big[\tange_P(W_1)\big]^\perp\subset\K_\alpha(P)$ or $\ls(P)$ has dimension one if $\big[\tange_P(W_1)\big]^\perp$ is transverse to $\K_\alpha(P)$.\\

\noi  Let $v_{_X}\in \ls(X)$ be a non zero vector. Extend this vector anchored at $X$ to a smooth vector field $\tilde{\cv}:K\to\tange\left(\Mgu\right)$ defined in a neighborhood $K\subset W_1$ of $X\in W_1$. Let $\pi_{_P}:\tange_P\left(\Mgu\right)\to\K_\alpha(P)$ be the orthogonal projection for $P\in{K}$. Consider the vector field defined on $K$ by $\cv_1(P)=\pi_{_P}\left(\tilde{\cv}(P)\right)$. \\

\noi Then $\cv_1$ is a smooth vector field and by continuity $\cv_1$ satisfies the required pro\-perty in a possible smaller neighborhood $\U$.\\

\noi Let $Y:K'\to\tange(W_1)\subset\tange\left(\Mgu\right)$ is a nonvanishing vector field, defined in a possible smaller neighborhood $K'$contained in $K$, tangent to the foliation $\mathcal{F}$ (in other words $Y(P)\in \Pi_1(P)$ for all $P\in K'$). Let $\varphi:K'\to\R$ be a smooth function that vanishes only at $X$ (again possible in an even smaller neighborhood).
Let $\cv_2=Y+\varphi\cv_1$. This is a vector field defined in a neighborhood of $X$ in $W_1$ having the property that
$\cv_2(X)=Y(X)$  and $\cv_2(P)\in\K_{\alpha}(P)$ but $\cv_2(P)\notin \tange (W_1)$ for all $P\not=X$.\\

\noi  To finish the proof of the lemma we have that by standard extension theorems (partition of unity) there exists an extension of $\cv_2$ to a nonsingular vector field $\cv_3:\V\to\tange\left(\Mgu\right)$ 
defined on an open neighborhood $\V\subset \Mgu$ of $X$. The vector field defined by $\cv(P)=\pi_{_P}\left(\cv_3(P)\right)$ has property that $\cv(P)\in\K_\alpha(P)$ for all $P\in\V$ and $\cv(P)=\cv_3(P)$ if $P\in K'$.\\

\noi By multiplying the vector field $\cv$ by a positive constant $c>0$, if necessary, we can assume that all the solutions of the differential equation defined by the vector field $c\cv$ on $\V$ are defined in the interval $(-2,2)$.\\

\noi If $\gamma:[-1,1]\to\U$ is the solution of the differential equation determined by $c\cv$ and satisfying the initial condition $\gamma(0)=X$, then this parametrized curve, if $c$ is sufficiently small, satisfies all the required properties.
 \end{proof}

\noi The open set  $V_1=\Mgu-W_1$ is connected since $W_1$ is of codimension two. This set is the set where the form $\alpha$ is of contact type and it follows by Darboux theorem that every two points in $V_1$ can be connected by a smooth curve with non vanishing tangent vector contained in $\K_\alpha$ (a regular Legendrian curve). \\

\noi By proposition \ref{P7} and lemma \ref{L6}  imply that in the moment-angle manifolds of mixed type $\Mgu$  the $1$-form $\alpha$ defines a \emph{conductive confoliation} in the sense of J. S. Altschuler and L. F.  Wu,  therefore $\alpha$ can be approximated in the $C^\infty$ topology by a contact form.  We have the following theorem:

\begin{theo}\cite[Theorem 12]{BLV}\label{T3}
Let $m=1$, $n>3$, $s=1$ and let $\Lam=\big(\lambda_1,\ldots, \lambda_n\big)$, with $\lambda_j\in\C$, be an admissible configuration. The manifolds 
$$
\Mgu=\underset{j=1}{\overset{2\ell+1}{\sharp}}\left(\s^{2d_j}\times\s^{2n-2d_j-1}\right),$$
where $d_j=n_j+\dots+n_{j+\ell-1}$,  admit contact structures.
\end{theo}

\noi Let now $m=1$, $n>3$ and $s>1$.

\begin{prop}\label{P8}
Let $m=1$, $n>3$ and $s>1$. Let $\ast$ denote the Hodge operator for a given Riemannian metric on a moment-angle manifold of mixed type $\Mg$. Then, for the appropriate orientation of $\Mg$, $\alpha$ is a positive confoliation:
\begin{itemize}
\item For $X=\big(w_1,\dots,w_s,z_1,\dots,z_n\big)\in\Mg$ such that $w^2_1+\dots+w^2_s\not=0$,
$$
\ast\Big(\alpha\wedge(d\alpha)^{n+s-2}\Big)(X)>0.
$$
\item For $X\in W_s$ such that $w^2_1+\dots+w^2_s=0$,
$$
\ast\Big(\alpha\wedge(d\alpha)^{n+s-2}\Big)(X)=0.
$$
\item The set of points $X\in\Mg$ such that $\ast\Big(\alpha\wedge(d\alpha)^{n+s-2}\Big)(X)=0$ is the real analytic set of real codimension two in $\Mg$ given by 
$$
W=\biggr\{X\in\Mg\;\big|\; w^2_1+\dots+w^2_s=0\biggl\}.
$$
\end{itemize}
\end{prop}

\noi  Let denote by $\tau$ the $2$-form  $\ast\Big(\alpha\wedge(d\alpha)^{n+s-3}\Big)$.

\begin{lema}\label{L7}
Let $\vd$ be a compact $(2n+1)$-dimensional manifold. Let $\alpha$ be a positive confoliation, i.e., $\ast\Big(\alpha\wedge (d\alpha)^n\Big)(X)\geq0$ , where $\ast$ is the Hodge star operator for a given Riemannian metric. Suppose that the set 
$$
S=\biggr\{X\in\vd\;|\; \ast\Big(\alpha\wedge (d\alpha)^n\Big)(X)=0\biggl\},
$$ 
is real analytic set of real codimension at least two, then $\alpha$ is a conductive confoliation.
\end{lema}

\noi The proof of this lemma is completely analogous to the proof of lemma \ref{L6} using the fact that $S$ has a Whitney stratification.

\begin{proof}
\noi Given a Whitney stratification of $S$, by Whitney conditions $A$ and $B$, the intersection of $\ker(\alpha)$ with any limit of tangent spaces of the strata of maximal dimension is a union $U$ of linear subspaces of dimension at most $2n-2$.\\

\noi Then there exist a vector $Y\in\tange_X(\vd)$ such that $Y\notin U$. Therefore there exist a vector field $\cv:U\to \tange(\vd)$ in a neighborhood $U$ of $X$ such that $\cv(X)=Y$. Let $\pi_P:\tange_P(U)\to \ker(\alpha_P)$ be the orthogonal projection for $P\in U$.\\

\noi Consider the vector field in $U$ defined by $\cvd(P)=\pi_P(\cv(P))$. At this point we can proceed exactly as in the last two paragraphs of the proof of lemma \ref{L6} to conclude that there exist a smooth parametrized curve $\gamma:[-1,1]\to\vd$ such that
\begin{itemize}
\item $\gamma(0)=Y$,
\item If $x\in[-1,1]$ with $x\not=0$, $\gamma(x)\notin S$ and $\gamma'(x)\in\big[\ker(\tau)\big]^\perp\big(\gamma(x)\big)=\K_\alpha\big(\gamma(x)\big)$.
\end{itemize}

\noi  Now we consider the curve $\hat{\gamma}:[-1,1]\to\vd$ given by $\hat{\gamma}(t)=\gamma(t^2)$ so that $\hat{\gamma}'(0)=0$. Then $\hat{\gamma}$ has the property that $\hat{\gamma}'(t)$ is in $\big[\ker(\tau)\big]^\perp\big(\gamma(x)\big)$ for all $t\in[-1,1]$ since $\big[\ker(\tau)\big]^\perp\big(\gamma(0)\big)=0$.
\end{proof}

\noi The open set  $V_W=\Mg-W$ is connected since $W$ is of codimension two. This set is the set where the form $\alpha$ is of contact type and it follows by Darboux's theorem that every two points in $V_W$ can be connected by a regular Legendrian curve. \\

\noi Then by proposition \ref{P8} and lemma \ref{L7}  it follows that in the moment-angle manifolds of mixed type $\Mg$  the $1$-form $\alpha$ defines a \emph{conductive confoliation} in the sense of J. S. Altschuler and L. F.  Wu,  therefore $\alpha$ can be approximated, in the $C^\infty$ topology,  by a contact form.  We have the following theorem:

\begin{theo}\label{T4}
Let $m=1$, $n>3$, $s>1$ and let $\Lam=\big(\lambda_1,\ldots, \lambda_n\big)$ be an admissible configuration, with $\lambda_j\in\C$. The manifolds 
$$
\Mg=\underset{j=1}{\overset{2\ell+1}{\sharp}}\left(\s^{2d_j+s-1}\times\s^{2n-2d_j+s-2}\right),$$
where $d_j=n_j+\dots+n_{j+\ell-1}$,  admit contact structures.
\end{theo}

%\begin{rmk}
 \noi  C. Meckert in \cite{Meckert} have shown that the connected sum of contact manifolds of the same dimension is a contact manifold. It was pointed to us by D. Pancholi that more generally it follows also from work of Y. Eliashberg and A. Weinstein (\cite{Elias},  \cite{Weins}) that the manifolds $\Mgu$ have a contact structure.\\

\noi Indeed, the manifolds $\Mgu$ are connected sums of products of the form $\s^n\times\s^m$ with $n$ even and $m$ odd, and $n,m>2$. Without loss of generality, we suppose that $m>n$ (the other case is analogous), then $\s^m$ is an open book with binding $\s^{m-2}$ and page $\R^{m-1}$. Hence $\s^n\times\s^m$ is an open book with binding $\s^{m-2}\times\s^n$ and page $\R^{m-1}\times\s^{n}$. The page  $\R^{m-1}\times\s^n$ is parallelizable since it embeds as an open subset of $\R^{m+n-1}$, therefore, since $m+n-1$ is even it has an almost complex structure.\\

\noi  Furthermore, by hypothesis, $2n\leq{n+m}$ hence by a theorem of Y. Eliashberg (see \cite{Elias}) the page is Stein and is the interior of a compact manifold with contact boundary $\s^{m-2}\times\s^n$. Hence by  a theorem of E. Giroux (see \cite{Giroux}) $\s^n\times\s^m$ is a contact manifold. \\

\noi However our construction is in some sense explicit  since it is the instantaneous  di\-ffusion through the heat flow of an explicit 1-form which is a positive confoliation.

\begin{prop}\label{P9}
Let $m>1$, $n>3$ and $n>2m$.  Let $\ast$ denote the Hodge ope\-rator for a given Riemannian metric on a moment-angle manifold of mixed type $\MG$. Then for the appropriate orientation of $\MG$  $\alpha$ is a positive confoliation:\begin{enumerate}
\item for $X=\big(w_1,\dots,w_m,z_1,\dots,z_n\big)\in\MG$ such that $w_k\not=0$ for all $k\in\{1,\dots,m\}$,
$$
\ast\Big(\alpha\wedge (d\alpha)^{n-1}\Big)(X)>0,
$$
\item for $X\in W_\ell$ with $\ell\in\{1,\dots,m\}$,
$$
\ast\Big(\alpha\wedge (d\alpha)^{n-1}\Big)(X)=0.
$$
\item The set of points $X=\big(w_1,\dots,w_m,z_1,\dots,z_n\big)$ in $\MG$ such that  
$$
\ast\Big(\alpha\wedge (d\alpha)^{n-1}\Big)(X)=0,
$$
is the real analytic set of real codimension two in $\MG$ given by:
$$
\Sigma=\biggr\{\big(w_1,\dots,w_m,z_1,\dots,z_n\big)\in\MG\;\Big|\; w_1\dots w_m=0\biggl\}.
$$
\end{enumerate}
\end{prop}

\noi  Let denote by $\tau$ the $2$-form  $\ast\Big(\alpha\wedge(d\alpha)^{n-2}\Big)$.\\

\noi To proof that $\alpha$ is a conductive confoliation when $m>1$, $n>3$ and  $n>2m$, we use the following lemma:

\begin{lema}\label{L8}
Let $\vd$ be a compact $(2n+1)$-dimensional manifold. Let $\alpha$ is a positive confoliation, i.e., $\ast\Big(\alpha\wedge (d\alpha)^n\Big)(X)\geq0$ , where $\ast$ is the Hodge star operator for a given Riemannian metric. Suppose that the set 
$$
S=\biggr\{X\in\vd\;\Big|\; \ast\Big(\alpha\wedge (d\alpha)^n\Big)(X)=0\biggl\},
$$ 
is real analytic set of real codimension at least two, then $\alpha$ is a conductive confoliation.
\end{lema}

\noi The proof of this lemma is completely analogous to the proof of lemma \ref{L6} and lemma \ref{L7}, using the fact that $S$ has a Whitney stratification.\\

%\begin{proof}
% Given a Whitney stratification of $S$, by Whitney conditions A and B, the intersection of $\ker(\alpha)$ with
%any limit of tangent spaces of the strata of maximal dimension is a union $U$ of linear subspaces  of dimension at most $2n-2$.\\
%
%\noi Then there exists a vector $Y\in\tange_X(\vd)$ such that $Y\notin{U}$. Therefore there exist a vector field $\cv:\U\to\tange(\vd)$  in a neighborhood $\U$ of $X$ such that $\cv(X)=Y$. Let $\pi_P:\tange_P(U)\to\ker(\alpha_P)$ be the orthogonal projection for $P\in{U}$. \\
%
%\noi Consider the vector field on $U$ defined by $\cvd(P)=\pi_P(\cv(P))$. At this point we can proceed exactly as in the last two paragraphs of the proof of lemma \ref{L6} to conclude that there exists a smooth parametrized curve
%$\gamma:[-1,1]\to\vd$ such that 
%\begin{enumerate}
%\item $\gamma(0)=Y$, 
%\item If $x\in[-1,1]$ with $x\not=0$, $\gamma(x)\notin{S}$  and $\gamma'(x) \in\big[\ker(\tau)\big]^\perp\big(\gamma(x)\big)=\K_\alpha\big(\gamma(x)\big)$.
%\end{enumerate}
%
%\noi Now we consider the curve $\hat{\gamma}:[-1,1]\to\vd$ given by $\hat{\gamma}(t)=\gamma(t^2)$ so that $\hat{\gamma}'(0)=0$. Then $\hat{\gamma}$ has the property that $\hat{\gamma}'(t)$ is in $\big[\ker(\tau)\big]^\perp\big(\gamma(t)\big)$ for all $t\in[-1,1]$ since  $\big[\ker(\tau)\big]^\perp\big(\gamma(0)\big)=0$.
%\end{proof}

\noi The open set  $V_\Sigma=\MG-\Sigma$ is connected since $\Sigma$ is of codimension two. This set is the set where the form $\alpha$ is of contact type and it follows by Darboux's theorem that every two points in $V_\Sigma$ can be connected by a regular Legendrian curve. \\

\noi Hence proposition \ref{P9}  and lemma \ref{L8} imply that in the moment-angle manifold of mixed type $\MG$ the $1$-form $\alpha$ defines a \emph{conductive confoliation} in the sense of J. S. Altschuler and L. F.  Wu
and therefore $\alpha$ can be approximated in the $C^{^\infty}$ topology by a contact form.\\

\noi  We summarize all in the following result:

\begin{theo}
Let $m>1$, $n>3$ such that $n>2m$ and let $\Lam=\big(\lam_1,\dots,\lam_n\big)$ be an admissible configuration in $\C^m$. The moment-angle manifolds of mixed type $\MG$ associated to $\Lam$,  are contact manifolds. The distribution $\K_\alpha$ can be perturbed by an arbitrarily small perturbation in the $C^\infty$ topology to a distribution 
$\K_{\alpha'}$ which defines a contact structure on $\MG$. 
\end{theo}

\section{Concluding Remarks.}\label{S4}

\noi Recall, for $m>1$ and $n>3$  and $n>2m$ the submanifold $\MGo$ in $\C^{m+n}-\big\{(0,\dots,0)\big\}$, which was defined before by the equation  \begin{equation}\label{E22}
\pmb{w}^2+\sum_{j=1}^n\lam_j|z_j|^2=0 
\end{equation}
\noi where $\pmb{w}^2=\big(w_1^2,\dots,w_m^2\big)\in\C^m$  with $w_k\in\C$ ($1\leq k \leq{m}$),  $\lam_j=\big(\lambda^1_{_j},\dots,\lambda^m_{_j}\big)\in\C^m$ for all $j\in\{1,\dots,n\}$ and $\Lam=\big(\lam_1,\dots,\lam_n\big)$ is an admissible configuration.\\

\noi $\MGo$ is indeed a manifold of dimension $2n$ since the only singularity of the variety given by \eqref{E22} is the origin. If 
$\big(w_1,\dots,w_m,z_1,\dots,z_n\big)$ satisfies \eqref{E22} then for every real number $t$ the point $\big(tw_1,\dots, tw_m,tz_1,\dots,tz_n\big)$ also satisfies \eqref{E22} therefore
if we add the origin we obtain a real cone with vertex the origin and we have $\MG=\MGo\cap\s^{2n+2m-1}$.\\ 

\noi The following proposition shows that in fact there is a relation between classical and moment-angle manifolds of mixed type, when $m>1$.

\begin{prop}
Let $m>1$, $n>3$ and $n>2m$.    
Let $\hat{p}:\MGo\to\C^n-\big\{(0,\dots,0)\big\}$ be the map given by 
$$
\hat{p}\left(w_1,\dots,w_m,z_1,\dots,z_n\right)=\left(z_1,\dots,z_n\right).
$$

\noi Then $\hat{p}$ is a differentiable branched covering of order $2^m$. It is  singular precisely in the set 
$\tilde\Sigma=S\cap{\MGo}$ where 
$S=\left\{\left(w_1, \dots, w_m, z_1,\dots,z_m\right) \; |\;  w_1w_2\dots {w_m}=0\right\}.$
\noi Furthermore, $\hat{p}$ induces a differentiable branched covering 
$$
p:\MG\to\s^{2n-1}$$ 
\noi of order $2^m$ which is singular at $\Sigma=S\cap{\MG}$. 
 \end{prop}

\begin{proof} Clearly  $\hat{p}$ is surjective since we can choose $\big(z_1,\dots,z_n\big)\neq0$ arbitrarily and then
find $\big(w_1,\dots,w_m\big)$ from \eqref{E22} so that $\big(w_1,\dots,w_m,z_1,\dots,z_n\big)\in\MGo$. Direct calculation shows
that if a point $X=\big(w_1,\dots,w_m,z_1,\dots,z_n\big)\in\MGo$ is such that $w_1w_2\dots{w_m}\neq0$ then $\hat{p}$ is a local diffeomorphism in a neighborhood of $X$. \\

\noi If $\hat{p}^{-1}\big(z_1,\dots,z_n\big)$ contains a point $X=\big(w_1,\dots,w_m,z_1,\dots,z_n\big)$ such that 
$w_j\neq0$ for exactly $\ell$ indices $j$ then the cardinality of $\hat{p}^{-1}\big(z_1,\dots,z_n\big)$ is $2^{\ell}$. \\

\noi Let us us recall that $\MGo\subset\C^{m+n}$ is a deleted real cone with vertex the origin of 
$\C^{n+m}$. For a point $X=\big(w_1,\dots,w_m,z_1,\dots,z_n\big)\in\MGo$ consider the ray
$$
R(X)=\biggr\{ t\big(w_1,\dots,w_m ,z_1, \dots,z_n\big)\;\big|\; t>0\biggl\}.
$$ 

\noi Then $\hat{p}$  sends bijectively $R(X)$ onto the ray through $\hat{p}(X)$:
$$
R\big(\hat{p}(X)\big)=\biggr\{t\big(z_1,\dots,z_n\big)\;\big|\;t>0\biggl\},
$$
\noi then the map :
$$
p(X)=|\hat{p}(X)|^{-1}\hat{p}(X), \quad\text{for} \quad X\in\MGo 
$$
\noi is the required  branched covering of order $2^m$.\\
\end{proof}

\noi There is another way to describe this branched covering. Let $G=\Z_2=\{1,-1\}$ be the multiplicative group with two elements. There exists a natural action of $G^m$ on $\MG$ as follows:
$$
T_{\left(\sigma_1,\dots,\sigma_m\right)}\left(w_1,\dots,w_m,z_1,\dots,z_n\right)=\left(\sigma_1{w_1},\dots,\sigma_m{w_m},z_1,\dots,z_n\right), 
$$
\noi where $\left(\sigma_1,\dots,\sigma_m\right)\in{G}^m$.\\

\noi It follows from equation \eqref{E22} that the set of fixed point of this action is $\MLamU$  (identified as the set of points in $\MG$ such that $w_j=0 $ for $j\in\{1,\dots,m\}$).\\

\noi The orbit space of the action of $G$ on $\MG$ is the sphere $\s^{2n-1}$. The subsets with isotropy group different from the identity are the submanifolds $M_K\subset\MG$ :
$$
M_K:=\left\{\left(w_1,\dots,w_m,z_1,\dots, z_n\right)\in\MG\;|\;    F_k\left(z_1,\dots,z_n\right)=0,\quad k\in K\right\},
$$
\noi where $\emptyset\neq{K}\subset\{1,\dots,m\}$ and the functions $F_k$ are the functions appearing in equation \eqref{E16} in proposition \ref{P4} which were used to define $\MG$. These manifolds are all moment-angle submanifolds of $\MG$. \\

\noi There is an action of the $n$-torus ${\mathbb T}^n=\s^1\times\dots\times\s^1$  on $\MG$ as follows:
$$
T_{\mathbf{u}}\left(w_1,\dots,w_m,z_1,\dots,z_n\right)=\left(w_1,\dots,w_m, u_1z_1,\dots,u_nz_n\right), \,\,\mathbf{u=}\left(u_1,\dots,u_n\right)\in\T^n.
$$

\noi Let us describe the orbit space of this action of $\T^n$ and the corresponding equivalent of the moment map. Define
$$
\frak{m}:\MG\to\C^m
$$

\noi by the formula

$$
\frak{m}\left(w_1,\dots,w_m,z_1\dots,z_n\right)=(w_1,\dots,w_m)
$$

\noi Define  $\frak{M}:\MG\to\C^m\times{\R}_+^n$ by the formula 
$$
\frak{M}\left(w_1,\dots,w_m,z_1\dots,z_n\right)=\left(w_1,\dots,w_m,|z_1|^2, \dots, |z_n|^2\right)
$$

\noi Let us recall that $m>1$ and $n>3$ and $2m<n$. If $\left(w_1,\dots,w_m, z_1,\dots, z_n\right)\in\MG$ we must have that 
$|z_1|^2+\dots+|z_n|^2\neq0$ since otherwise $w_r=0$ for all $r\in\{1,\dots,m\}$. Since $\MG$ is compact
there exists $0<c<1$ such that 
$$
c=\inf\left\{|z_1|^2+\dots+|z_n|^2\,|\,\left(w_1,\dots,w_m,z_1\dots,z_n\right)\in\MG\right\}.
$$

\noi It follows immediately from equations defining $\MG$ that :

$$
{P(\Lam)}:=\frak{m}\left(\MG\right)\subset\ev\left(c\lam_1,\dots,c\lam_n\right),
$$
\noi where $\Lam=\left(\lam_1,\dots,\lam_n\right)$ is the admissible configuration associated to $\MG$ and $\ev\left(c\lam_1,\dots,c\lam_n\right)$ is the convex hull in $\C^m$
of the set $\left\{c\lam_1,\dots,c\lam_n\right\}$,  so that  $\ev\left(c\lam_1,\dots,c\lam_n\right)= c\ev\left(\lam_1,\dots,\lam_n\right)$. \\

\noi Since by hypothesis $\Lam$ is admissible it follows that $c\ev\left(\lam_1,\dots,\lam_n\right)$ is 
a full polytope of dimension $2m$ with $n$ vertices. \\

\noi Its Gale transform $\hat{P}(\Lam)$ is a convex polytope
of dimension $n-2m-1$:
$$
\hat{P}(\Lam)=\left\{\left(t_1,\dots,t_n\right)\in\R^n\,|\, t_j\geq0 \,\text{for}\,\, j\in\{1,\dots,n\},\, \sum_{j=1}^nt_jc\lam_j=0,\,\,\sum_{j=1}^nt_j=1\right\}.
$$

\noi It is immediate to verify that if $\mathbf{w}\in P(\Lam)$ then $r\mathbf{w}\in P(\Lam)$ for all $r\in [0,1]$. Also there exists $\epsilon>0$ such that $\mathbf{w}\in P(\Lam)$ if $|\mathbf{w}|<\epsilon$. Hence we have:

\begin{prop}
$P(\Lam)$ is star-shaped with respect to the origin and with nonempty interior.
\end{prop}

\noi For $\mathbf{w}\in{P(\Lam)}$, consider the convex polytope of dimension $n-2m-1$:

$$
{P}_{\mathbf{w}}(\Lam)=\left\{\left(t_1,\dots,t_n\right)\in\R_+^n\,\,\,\big|\,\, \sum_{j=1}^nt_j\lam_j=\mathbf{w^2},\,\,\sum_{j=1}^nt_j=1-\sum_{k=1}^m|w_k|^2\right\}.
$$
\noi \\

\noi Then  $\frak{M}\left(\MG\right)$, the image of $\MG$ under $\frak{M}$, is given by:

$$
\mathbf{P}(\Lam)=\left\{(\mathbf{w},\mathbf{t}) \,\,|\,\,\mathbf{w}\in{P(\Lam)}=\frak{m}\left(\MG\right),\,\mathbf{t}=\left(t_1,\dots,t_n\right)\in{P}_{\mathbf{w}}(\Lam)\right\}.
$$
\noi A point in $\mathbf{P}(\Lam)$ determines a unique orbit of the torus action and to any orbit corresponds in the natural way a point in $\mathbf{P}(\Lam)$. Hence

\begin{prop} $\mathbf{P}(\Lam)$ is the orbit space of the action of $\T^n$ on $\MG$.
\end{prop}

\noi Therefore, $\frak{M}$ plays the role of the moment map
of the action and the topology of $\MG$ is completely described by the combinatorics of 
$\mathbf{P}(\Lam)$, $P(\Lam)$ and $\hat{P}(\Lam)$. \\

\noi From the above discussion we see that moment-angle manifolds of mixed type are also toric spaces (see \cite{BP}) and their topology is analogous to the moment-angle manifolds. Their topology is very rich and it is essentially reduced to combinatorics like cla\-ssical algebraic toric or quasi-toric varieties (see \cite{BP2}, \cite{BP} and \cite{BBCG}).\\

\noi It is an interesting problem to determine the topology of moment-angle manifolds of mixed type.  In analogy with the results of moment-angle manifolds des\-cribed in \cite{BM}, we state the following conjecture: \\

\noi {\bf Conjecture:} 
Given any sequence of finitely generated abelian groups $G_1,\dots, G_k$ there exist integers $m$, $n$ and an admissible configuration $\Lam$ such that for $j=1,\dots,k$ and some integer $\ell$,  $G_k$ is a direct summand of $H^{j+\ell+1}(\MG,\Z)$, where $\MG$ is the moment-angle manifold of mixed type  of dimension $N>k$  corresponding to $m$, $n$ and $\Lam$.\\

\noi In particular the conjecture and our results would imply that given any compact simplicial complex $K$ there exists a compact contact manifold $M$ such that $H^\ast(M,\Z)$ contains $H^\ast(K,\Z)$ as direct summand and in particular there would exist contact manifolds with arbitrary torsion in prescribed dimensions.\\

\noi Another aspect we would like to remark is that although the heat-method results of S. J. Altschuler and L. F. Wu are very beautiful, they are very difficult to implement except for dimension three and some special cases. Our method gives a series of nice examples were this method can actually be applied and it is possible that many other examples will be obtained by our  methods.\\

 \noi In conclusion: we have shown that moment-angle manifolds of mixed type give new examples of contact manifolds with very rich topology and in arbitrarily large dimensions.\\

\noi {\bf Acknowledgments.} We would like to thank professors Steven J. Altschuler ,  S. L\'opez de Medrano and  D. Pancholi for very useful observations and comments. In particular we are grateful with professor S. L\'opez de Medrano for suggesting the proof of proposition \ref{P4}.


\begin{thebibliography}{100}

\bibitem{Alt} Altschuler, Steven J, ``A geometric heat flow for one-forms on three dimensional manifolds", \emph{Illinois Journal of Mathematics}, Vol. 39, no. 1, pp. 98-118, 1995.

\bibitem{AltWu} Altschuler, Steven J. and Wu, Lani F., ``On Deforming Confoliations", \emph{J. Differential Geometry}, Vol. 54, pp. 75-97, 2000.

\bibitem{BBCG} Bahri, A. and  Bendersky, M. and Cohen, F. R. and Gitler, S.,  ``The polyhedral product functor: a method of decomposition for moment-angle complexes, arrangements and related spaces"., \emph{Adv. Math.}, Vol. 225, no. 3, pp. 1634-1668, 2010.

\bibitem{BLV} Barreto, Y. and L\'opez de Medrano S. and Verjovsky A., ``Open Book Structures on Moment-Angle Manifolds $Z^\C(\Lam)$ and Higher Dimensional Contact Manifolds", \emph{arXiv:1303.2671v1 [math.AT] 11 Mar 2013}.

 \bibitem{BM} Bosio, F and Meersseman, L.``Real Quadrics in $\C^n$, complex manifolds and convex polytopes"., \emph{Acta Math.}, Vol. 197, pp. 53-127, 2006.
       
\bibitem{BP} Buchstaber,V. M. and  Panov, T. E.  ``Torus actions and their applications in Topology and Combinatorics", \emph{ University Lecture Seriers, AMS}, 2002.

\bibitem{BP2} Buchstaber,V. M. and  Panov, T. E. `Toric Topology'', \emph{arXiv:1210.2368v2 [math.AT] 3 Mar. 2013}.

\bibitem{DJ} Davis, M. W, Januszkiewicz, T. ``Convex polytopes, coxeter orbifolds and torus actions''
\emph{Duke Math. Journal.} Vol. 62, No. 2,  pp. 417--451, 1991.

\bibitem{Elias} Eliashberg, Y., ``Topological characterization of Stein manifolds of dimension $>2$ ," \emph{International Journal of  Mathematical}, Vol. 1, No. 1, pp. 29-46,  1990.

\bibitem{ET} Eliashberg, Yakov M. and Thurston, W. P., ``Confoliations", \emph{American Mathematical Society. Lectures Series}, Vol. 13, 1998.

\bibitem{Giroux} Giroux, E., ``Geometrie de contact: de la dimension trois vers les dimensions superieures", \emph{ICM}, Vol. II, pp. 405-414, 2002.

\bibitem {LoGli} Gitler, S. and L\'opez de Medrano, S., ``Intersections of quadrics, moment-angle manifolds and connected sums", \emph{to appear in Geometry and Topology, 2013.}.

\bibitem{GL} G\'omez Gutierr\'ez, V. and L\'opez de Medrano, S., ``Topology of Intersections of Quadrics II", \emph{in preparation}. 

\bibitem{LN} Loeb J. J., Nicolau M., ``On the complex geometry of a class of non-K\"ahler manifolds", \emph{Journal: Israel Journal of Mathematics - ISR J MATH}, Vol. 110, no. 1 pp. 371-379, 1999.

\bibitem{SLM1} L\'{o}pez de Medrano, S., ``Singularities of homogeneous quadratic mappings", \emph{Revista de la Real Academia de Ciencias Exactas, F\'isicas y Naturales, Springer-Verlag.} Serie A. (DOI: 10.1007/s13398-012-0102-6). 

\bibitem{LV} L\'opez de Medrano, Santiago and Verjovsky, Alberto, ``A new family of complex, compact, non-symplectic manifolds",
\emph{Bol. Soc. Bras. Mat. }, Vol. 28, No. 2, pp. 253-269, 1997.

\bibitem{LM} Lutz, Robert and Meckert, Christiane, ``Structures de contact sur certaines sph\'eres exotiques", \emph{C. R. Acad. Sci. Paris S\'er. A-B }, Vol. 282, pp. A591-A593, 1976.


\bibitem{Mather} Mather, John, ``Notes on topological stability", \emph{available on his webpage at Princeton University}, 1970.

\bibitem{Meckert} Meckert, Christiane, ``Forme de contact sur la somme connexe de deux vari\'et\'es de contact de dimension impare", \emph{Annales De L'Institut Fourier} Vol. 32, No. 3 pp. 251-260, 1982.

\bibitem{Meer} Meersseman, L.,`` A new geometric construction of compact complex
manifolds in any dimension'', \emph{ Math. Ann.  }, 317, 79Ð115 (2000).

\bibitem{MV} Meersseman, Laurent and Verjovsky, Alberto,
``Holomorphic principal bundles over projective toric varieties",
\emph{J. reine angew, Math}, Vol. 572, pp. 57-96, 2004.

\bibitem{St}  Sternberg, Shlomo,  ``Lectures on differential geometry''. Second edition. With an appendix by Sternberg and Victor W. Guillemin.\emph{ Chelsea Publishing Co.}, New York, 1983. 

\bibitem{IV}  Vaisman,  Izu,  ``Lectures on the geometry of Poisson Manifolds''. \emph{ Birkh\"auser Verlag}, Germany, 1994. 

\bibitem{Weins} Weinstein, Alan,  ``Contact surgery and symplectic handlebodies''. \emph{Hokkaido Mathematical Journal}, 20 (2). 241, 1991.

\end{thebibliography}
\end{document}